\documentclass[12pt,a4paper]{amsart}
\usepackage[utf8]{inputenc}
\usepackage[T1]{fontenc}
\usepackage{lmodern}
\usepackage{amsmath,amstext,amssymb,amsopn,amsthm,color,graphicx,enumitem,mathrsfs,mathtools,dsfont}
\usepackage[margin=1in]{geometry}
\usepackage{float}
\usepackage[colorlinks,citecolor=blue,urlcolor=blue,bookmarks=true]{hyperref}
\usepackage{tikz}
\usetikzlibrary{intersections}
\usetikzlibrary{arrows.meta}
\usetikzlibrary{decorations.text}
\usetikzlibrary{decorations.pathmorphing}
\usetikzlibrary{positioning}
\usetikzlibrary{calc}

\usepackage{color}

\definecolor{kb}{rgb}{0.1,0.5,0.1}
\definecolor{tg}{rgb}{0.7,0.1,0.2}
\definecolor{lw}{rgb}{0.1,0.1,1}
\definecolor{ell}{rgb}{0.0,0.5,0}

\newcommand{\ind}{\mathbf{1}}
\newcommand{\vphi}{\varphi}

\numberwithin{equation}{section}

\newtheorem{theorem}{Theorem}[section]
\newtheorem{lemma}[theorem]{Lemma}
\newtheorem{proposition}[theorem]{Proposition}
\newtheorem{corollary}[theorem]{Corollary}

\theoremstyle{definition}

\newtheorem{example}[theorem]{Example}

\newcommand{\E}{\mathbb{E}}
\newcommand{\D}{\mathbb{D}}
\newcommand{\N}{\mathbb{N}}

\renewcommand{\P}{\mathbb{P}}
\newcommand{\R}{\mathbb{R}}
\newcommand{\Rd}{{\mathbb{R}^d}}

\newcommand{\bfX}{\mathbf{X}}
\newcommand{\bfY}{\mathbf{Y}}

\def\calD{{\mathcal{D}}}

\def\calL{{\mathcal{L}}}
\def\calR{{\mathcal{R}}}

\renewcommand{\leq}{\leqslant}
\renewcommand{\le}{\leq}
\renewcommand{\geq}{\geqslant}
\renewcommand{\ge}{\geq}

\DeclareMathOperator{\dist}{dist}

\DeclareMathOperator{\supp}{supp}
\newcommand{\norm}[1]{\|#1\|}

\newcommand{\cbhi}{C_{\mathrm{BHI}}}
\newcommand{\cbhit}{\widetilde{C}_{\mathrm{BHI}}}
\newcommand{\cnu}{C_{\text{\rm L\'evy}}}

\newcommand{\one}{{\bf 1}}

\renewcommand{\epsilon}{\varepsilon}

\usepackage{amsaddr}

\begin{document}

\title{Yaglom limit for unimodal L\'evy processes}

\author[G. Armstrong, K. Bogdan, T. Grzywny, \L{}. Le\.{z}aj and L. Wang]{Gavin Armstrong$^1$, Krzysztof Bogdan$^2$, Tomasz Grzywny$^2$, \L{}ukasz Le\.{z}aj$^2$ and Longmin Wang$^3$}
\address{$^1$Mathematics Department, Central Washington University, 400E University Way, Ellensburg, WA 98926, USA \\ \href{gavin.armstrong@cwu.edu}{gavin.armstrong@cwu.edu}}
\address{$^2$Department of Mathematics, Wroc\l{}aw University of Science and Technology, wyb. Wyspia\'{n}skiego 27, 50-370 Wroc\l{}aw, Poland \\ \href{mailto:krzysztof.bogdan@pwr.edu.pl}{krzysztof.bogdan@pwr.edu.pl}, \href{tomasz.grzywny@pwr.edu.pl}{tomasz.grzywny@pwr.edu.pl}, \href{lukasz.lezaj@pwr.edu.pl}{lukasz.lezaj@pwr.edu.pl}}
\address{$^3$School of Statistics and Data Science, Nankai University, Tianjin 300071, P.R. China \\  \href{wanglm@nankai.edu.cn}{wanglm@nankai.edu.cn}}

\thanks{G.A., T.G. and \L{}.L. were partially supported by the National Science Centre (Poland): grant 2016/23/B/ST1/01665. K.B. was partially supported by the National Science Centre (Poland):
grant 2017/27/B/ST1/01339. L.W. was partially supported by the National Natural Science Foundation of China: Grant 11801283}


\begin{abstract}
We prove universality of the Yaglom limit of Lipschitz cones among all unimodal L\'{e}vy processes sufficiently close to the isotropic $\alpha$-stable L\'{e}vy process.
\end{abstract}

\maketitle

 %
 %

\section{Introduction}\label{sec:i}

Let $d \geq 1$ and $\bfX=(X_t\colon t \geq 0)$ be a pure-jump isotropic unimodal L\'{e}vy process in $\R^d$ with non-integrable L\'{e}vy density $\nu$ and the characteristic exponent $\psi$. Thus,
\begin{equation}\label{eq:kb05}
\E e^{i \xi \cdot X_t} = 
e^{-t\psi(\xi)}, \quad \xi \in \Rd,\ t\ge 0,
\end{equation}
where
\begin{equation}\label{eq:kb03}
	\psi(\xi) = \int_{\Rd} \big(1-\cos \xi \cdot z\big)\nu(z)\,dz, \quad \xi \in \Rd,
\end{equation}
\begin{equation}\label{eq:kb04}
	\int_{\Rd} \nu(z)\,dz=\infty \quad \text{and} \quad \int_{\Rd} \big(|z|^2 \wedge 1\big)\nu(z)\,dz< \infty,
\end{equation}
and $\nu(z)$ is a radial function, non-increasing in $|z|$. For normalization we also assume that $\psi(1)=1$. We let $\Gamma$ be an arbitrary Lipschitz cone on $\Rd$ with the vertex at the origin and we denote by $\tau_{\Gamma}$ the first exit time from the cone,
\[
\tau_{\Gamma} = \inf \{t>0 \colon X_t \notin \Gamma \}.
\]
Our main objective is to determine under which conditions there exist a function $g \colon [0,\infty) \mapsto (0,\infty)$ and a probability measure $\mu$ on $\Gamma$ such that
\begin{equation}\label{eq:yaglom}
\lim_{t \to \infty} \P_x  \left( \frac{X_t}{g(t)} \in A \big| \tau_{\Gamma}>t \right) = \mu(A), \quad A \subset \Gamma.
\end{equation}

The above limit, if it exists,  is called the Yaglom limit of $\Gamma$ for $\bfX$, after Akiva Moiseevich Yaglom, who first identified quasi-stationary distributions for the subcritical Bienaym\'{e}-Galton-Watson branching processes in \cite{Yaglom47}. Generally speaking, Yaglom limits describe the large-time limiting behaviour of processes conditioned not to become extinct or absorbed. Accordingly, the random time which defines the conditioning in \eqref{eq:yaglom} and in similar expressions is sometimes referred to as absorption time.
For example, in the one-dimensional case, the absorption time is usually the first hitting time of the origin, see, e.g., Haas and Rivero \cite{HassRivero12}, or the first hitting time of the negative half-line $(-\infty,0)$, see, e.g., A. Kyprianou and Palmowski \cite{KyprianouPalmowski06}.
Over the last years the interest in the existence of Yaglom limits and quasi-stationary distributions {was steadily growing}. {This is due to the fact that the topic is mathematically challenging and in various applications it is natural to ask about the limiting behaviour conditional on non-extinction. 
{However}, so far the existence of the Yaglom limit for L\'{e}vy processes {was studed mainly} in the one-dimensional setting {or in the finite volume setting}, 
{with the notable exception of} the result Bogdan, Palmowski and Wang \cite{BoPaWa2018} explained below, which we extend in this work.

We write $\nu(r):=\nu(x)$ if $|x|=r$ for some $x \in \Rd$ (the radial profile of the L\'evy density).
The following additional assumptions on $\bfX$ will be explicitly made when needed:
\begin{itemize}
\item[{\bf A1}] There is $\beta \in (0,2)$ and $M>1$ such that
\begin{equation}\label{eq:scaling}
 \frac{\nu(r_1)}{\nu(r_2)} \le M \left(\frac{r_1}{r_2}\right)^{\!\!-d - \beta},\quad 0 < r_1 < r_2 < \infty.
\end{equation}
\item[{\bf A2}] $\nu(r)$ is regularly varying at infinity with index $-d-\alpha$ for some $\alpha \in (0,2)$. 
\item[{\bf A3}] There exists $t_0> 0$ such that $e^{-t_0\psi}\in L^1$.
\end{itemize}
If {\bf A1}, {\bf A2} and {\bf A3}  are satisfied, then we say in short that {\bf A} holds. We remark that the condition {\bf A1} is important in the probabilistic potential theory, in particular for estimating the Dirichlet heat kernel for $\bfX$, see Bogdan, Grzywny and Ryznar \cite{KBTGMR-dhk}. Assumption {\bf A2} is crucial for asymptotic analysis, in particular it puts the one-dimensional distributions of $\bfX$ in the domain of attraction of the isotropic $\alpha$-stable law. By Cygan, Grzywny and Trojan \cite[Theorem 7 and Proposition 2]{WCTGBT},  {\bf A2} is equivalent to $\psi$ being regularly varying at $0$ with index $\alpha$. The condition {\bf A3} allows to formulate our results in terms of convergence of probability densities.  Noteworthy, by Knopova and Schilling \cite[Lemma 2.6]{MR3010850} and \eqref{eq:kb05}, {\bf A3} is equivalent to the boundedness  of the transition density of $\bf{X}$ for some (large) time $t$.

For example, the geometric $\alpha$-stable L\'evy process with $\alpha\in(0,2)$ satisfies {\bf A}, since {\bf A2} and {\bf A3} follow from the formula $\psi(x)=\log(1+|\xi|^\alpha)$ and {\bf A1} follows from Grzywny and Ryznar \cite[p. 10]{MR3007664}; see Grzywny, Ryznar and Trojan \cite{TGMRBT} and 
\v{S}iki\'{c}, Song and Vondra\v{c}ek \cite{MR2240700} for further properties of this important process.

The focal point of our discussion, however, is the isotropic $\alpha$-stable process in $\Rd$ with $\alpha\in(0,2)$, which we denote by $\bfX^{\alpha}=(X_t^{\alpha}\colon t \geq 0)$ \cite{KBTGMR-dhk}, see also below. Of course, the L\'evy-Khinchine exponent of  $\bfX^{\alpha}$ is $\psi(\xi)=|\xi|^{\alpha}$ and the process satisfies {\bf A}.
For the sake of clarity and consistency, the objects pertaining to $\bfX^{\alpha}$ will be marked by a superscript~$\alpha$, which is then not an exponent. For instance, we write $\tau_{\Gamma}^{\alpha}=\tau_{\Gamma}(\bfX^{\alpha})$. Let $\mu^{\alpha}$ be the Yaglom limit of the isotropic $\alpha$-stable process for the Lipschitz cone $\Gamma$, given in \cite[Theorem 1.1]{BoPaWa2018}:
\begin{equation}\label{e.Ysl}
\lim_{t \to  \infty}\P_x \big( t^{-1/\alpha}X_t^{\alpha} \in A | \tau_{\Gamma}^{\alpha}>t \big) = \mu^{\alpha}(A), \quad A \subset \Gamma,\ x \in \Gamma.
\end{equation}
 By \cite[Theorem 3.3]{BoPaWa2018}, the probability measure $\mu^{\alpha}$ has density $n^{\alpha}$ with respect to the Lebesgue measure and the following convergence of densities holds,
	\begin{equation}\label{eq:limit_n}
	n^{\alpha}(y) = \lim_{\Gamma \ni x \to 0} \frac{p_1^{\Gamma,\alpha}(x,y)}{\P_x(\tau_{\Gamma}^{\alpha}>1)},\quad y\in \Gamma,
	\end{equation}
$p_t^{\Gamma,\alpha}$ being the Dirichlet heat kernel of the cone $\Gamma$ for $\bfX^{\alpha}$, see \eqref{eq:Hunt}. 

	Here is the main result of the paper. As usual, we let
	\[
	\psi^{-1}(u) = \sup \{ r\ge 0 \colon \psi(r)=u \}, \quad u\ge 0.
	\]
\begin{theorem}\label{thm:Yaglom}Let $X$ be a pure-jump isotropic unimodal L\'{e}vy process with characteristic exponent $\psi$.
Assume {\bf A}. Let $\Gamma$ be a Lipschitz cone with vertex at the origin. Then,
	\begin{equation}\label{e.Yl}
	\lim_{s \to \infty} \P_x \big( \psi^{-1}(1/s)X_s\in A | \tau_{\Gamma}>s \big) = \mu^{\alpha}(A),\quad x\in \Gamma,\ A\subset \Rd.
	\end{equation}
\end{theorem}
For example, the \textit{rescaling} factor is $\psi^{-1}(1/s)=(e^{1/s}-1)^{1/\alpha}$ for the geometric $\alpha$-stable L\'evy process. A qualitatively different rescaling can be found in Example~\ref{e.da}, so the \textit{universality} of the Yaglom limit for the considered class of processes refers to the limit but not the rescaling. 
The proof of Theorem \ref{thm:Yaglom} is given at the end of Section~\ref{sec:Martin}. In fact, the whole paper is devoted to this proof, but the setting and some auxiliary results may be of independent interest, e.g., the fact that in general $\bfX$ is \textit{not} self-similar but we have pointwise convergence of  density functions and uniformity of potential-theoretic estimates and asymptotics for the \textit{rescaled processes} $(\psi^{-1}(1/s)X_{st} \colon t > 0)$ parametrized by $s\ge 1$. We note in passing that  \eqref{e.Yl} also holds for $\Gamma=\Rd$, as a consequence of Lemma \ref{lem:hk_conv} and Scheff\'{e}'s lemma, which may also be of interest. 

Let us discuss in more detail some classical results on quasi-stationary distributions. {The seminal work of Yaglom appeared in 1947.} 
Since then a {huge} number of papers {have dealt with this problem}; a comprehensive {survey is given by} Pollet  \cite{Pollet08}. {In particular,} the discrete-time Markov chains were analysed by Seneta and Vere-Jones \cite{MR207047} in 1967 and by Tweedie \cite{Tweedie74} in 1974. In 1995 Jacka and Roberts \cite{JackaRoberts95} studied the continuous-time case. In 1971 E.~Kyprianou \cite{AEK71} considered a conditional limit distribution {for} the virtual waiting time in queues with Poisson arrival. Quasi-stationary distributions for Markov chains {on non-negative integers} with {absorption at} the origin {were considered by} Ferrari et al. \cite{MR1334159}, Flaspohler and Holmes \cite{MR346932}, \cite{MR207047} and van Doorn \cite{MR1133722}.

In general, studies on quasi-stationary distributions {are model specific ---} each class of stochastic processes {is} treated in a different way. {Let us further mention the quasi-birth-and-death processes (see, e.g., Bean et al. \cite{BBLPPT, BPT98, BPT00}), Fleming-Viot processes (Asselah et al. \cite{MR3498004} or Ferrati and Mari\'{c} \cite{MR2318407}) and branching processes ({Lambert \cite{MR2299923}}, see also Ren et al. \cite{MR3841403, MR4102269} and Harris et al. \cite{HHKW} for the so-called critical neutron transport), which are the continuous-space-time analogues of the Bienaym\'{e}-Galton-Watson process.} {Finally}, Bean et al. \cite{BOP19} gave Yaglom limits for stochastic fluid models.

	Let us turn our attention {back}  to L\'{e}vy processes. The case of the Brownian motion with drift was resolved by Mart\'{\i}nez and San Mart\'{\i}n \cite{SMJS94} in 1994. {The result} is an analogue of the random walk case studied by Iglehart \cite{MR368168} in 1974. The case of {the so-called} spectrally one-sided L\'{e}vy processes was investigated  in \cite{KyprianouPalmowski06} in 2006, using the Wiener-Hopf factorisation. The 
absorption time in \cite{KyprianouPalmowski06} is taken as the first ruin time:
\[
\tau_0^- = \inf \{t \geq 0 \colon X_t<0\}.
\]
A similar problem was investigated by Czarna and Palmowski \cite{MR3648297} for the Parisian ruin, 
\[
\tau^\theta = \inf \left\{ t>0\colon t-g_t>\textrm{\bf e}_{\theta}^{g_t} \right\}.
\]
Here $g_t = \sup \{s \leq t \colon X_s \geq 0\}$ and $\mathrm{\bf e}_\theta^{g_t}$ is an independent exponential random variable with intensity $\theta>0$, so the ruin occurs if  negative {values of $X$ persist longer than} ${\rm Exp}(\theta)$.

The spectrally one-sided L\'{e}vy processes were also considered by Mandles et al. in \cite{MR2959448} and Palmowski and Vlasiou \cite{MR4171931} provided the speed of convergence to the quasi-stationary distribution, which turned out surprisingly slow.

Limits similar to \eqref{eq:yaglom} for self-similar Markov processes in the one-dimensional case were studied by Haas and Rivero \cite{HassRivero12}. 
{Note that the normalization by $g$ in \cite{HassRivero12} strongly depends on the tails of the L\'{e}vy measure.}

To the best of our knowledge, the only study of Yaglom limits for unbounded sets in the multidimensional case is \cite{BoPaWa2018}, except for results similar to Zhang et al. \cite{MR3247530} on Markov processes killed upon leaving a set of bounded volume. The volume boundedness allows to employ the first eigenfunction, however it is a prohibitive restriction in our setting.

Finally, let us note that the conditioning defined by \eqref{eq:yaglom} is related to but different than the {conditioning} of the process to stay {forever} in {a} set; we refer to Bertoin and Doney \cite{MR1331218} for the case of random walks and to Bertoin \cite{MR1232850}, Chaumont \cite{MR1419491} and Chaumont and Doney \cite{MR2164035} {for the case of} continuous-time processes.

The proof of Theorem~\ref{thm:Yaglom} is inspired by \cite{BoPaWa2018}, the limit resulting from boundary asymptotics of Green potentials, but there are fundamental difficulties to overcome, mostly the lack of self-similarity (exact scaling) of $\bfX$, low regularity of the heat kernel of the process, highly non-trivial uniform in $s\ge 1$ estimates for the rescaled processes 
$(\psi^{-1}(1/s)X_{st} \colon t > 0)$, 
and delicate considerations related to continuity of their random functionals, mainly the first exit time of $\Gamma$, in the Skorokhod topology as $s\to \infty$. A crucial step in our development is Theorem \ref{thm:uni_lim}, which establishes \textit{stability}, or \textit{uniformity}, of limits of ratios of harmonic functions at the boundary of the cone. Namely, we refine the already general statement of Kwa\'{s}nicki and Juszczyszyn \cite{KJ17}, by proving that for the rescaled processes the ratios converge uniformly. Furthermore, we avoid superfluous technical assumptions on $\nu$ -- this makes the arguments on convergence of integrals tricky, but statements simple.  Of course, a number of further questions is now open: the case of non-Lipshitz cones $\Gamma$, the Yaglom limit for anisotropic stable L\'evy processes (even for processes with independent components) and universality classes for Yaglom limits in each such case.

The structure of the paper is as follows. In Section \ref{sec:prelims} we discuss definitions and basic results.
In Section~\ref{s.hke} we give estimates of the (Dirichlet) heat kernel of $\Gamma$ for the process $\bfX$.
In Section~\ref{s.rp} we analyse the effects of rescaling of $\bfX$, including the convergence of den\-sities of the rescaled processes.
In particular, Section~\ref{s.ubhp} is devoted to estimates of harmonic functions and oscillation-reduction, uniform for the rescaled processes, and in Section~\ref{s.ui} we verify the convergence of the survival probabilities. In Section~\ref{sec:Martin} we obtain normalized limits of Green potentials, expressed in terms of the Martin kernel, which we then bootstrap to the level of the heat kernel and  the heat kernel conditioned by the survival probability, to finally prove Theorem~\ref{thm:Yaglom} using  the uniformity of relative estimates and asymptotics with respect to the scaling parameter $s$, as $s\to \infty$.

\section{Notation and preliminaries}\label{sec:prelims}
By $\calR_0^{\alpha}$ we denote the class of (positive) functions {$f \colon (0,\infty) \mapsto (0,\infty)$} which are regularly varying at the origin with index $\alpha$. Thus, $f \in \calR_0^{\alpha}$ if for every $\lambda>0$,
\[
\lim_{x \to 0^+} \frac{f(\lambda x)}{f(x)} = \lambda^{\alpha}.
\]
Furthermore, we say {that} $f$ is regularly varying at infinity with index $\alpha$ and we write $f \in \calR_{\alpha}^{\infty}$ if for {every} $\lambda>0$,
\[
\lim_{x \to \infty} \frac{f(\lambda x)}{f(x)} = \lambda^{\alpha}.
\]

{Notation $c=c(a,b,\ldots)$ means that number $c\in (0,\infty)$, called \textit{constant}, may be so chosen to depend only on $a,b,\ldots$.} {For non-negative functions $f,g$} we use the notation $f \lesssim g$ if there is a constant $c > 0$ such that $f \leq c g$ {and we write} $f \stackrel{c}{\approx} g$ if $c^{-1} g \leq f \leq c g$. If the constant $c$ is not important, we may write $f \approx g$. {As usual}, $B(x,r)=\{y \in \Rd \colon |y-x|<r\}$ is {the open} ball in $\Rd$ with radius $r>0$ and center $x\in \Rd$. We let $B_r:=B(0,r)$. For $x \in D\subset \Rd$ we let $\delta_{D}(x) = \inf \{ |y-x|\colon y \in D^c \}$. All the considered sets, measures and functions are tacitly assumed Borel.

As in Introduction, $\bfX = (X_t\colon t\geq 0)$ is a non-trivial pure-jump isotropic unimodal L\'{e}vy process in $\Rd$ with $d \geq 1$. The L\'{e}vy density $\nu$ is non-integrable but satisfies the L\'{e}vy measure condition in {\eqref{eq:kb04}} and the L\'{e}vy-Khintchine exponent $\psi$ is given by {\eqref{eq:kb03}}.
The L\'evy-Khinchine formula {\eqref{eq:kb05} relates the distribution of $X_t$ to $\psi$ (and $\nu$).}
Since $\bfX$ is isotropic unimodal and has infinite L\'{e}vy measure, the distribution of $X_t$ is absolutely continuous if $t>0$. We denote the density by $p_t$. {In fact}, due to {Kulczycki and Ryznar} \cite[Lemma 2.5]{MR3413864}, $p_t$ is continuous on $\mathbb{R}^d\setminus\{0\}$. 
We define
$$p_t(x,y)=p_t(y-x) \quad \mbox{and}\quad \nu(x,y)=\nu(y-x),\quad t>0,\, x,y\in \Rd,$$ 
the heat kernel  (the transition probability density) and the jumping kernel of $\bfX$, {correspondingly}. Of course, $p_t(x,y)=p_t(y,x)$ and $\nu(x,y)=\nu(y,x)$. Let
\[
P_t(x,A) = \int_A p_t(x,y)\,dy, \quad t>0,\ x \in \Rd, \ A \subset \Rd.
\]
It is the transition {probability} of $\bfX$. {We also define the operator semigroup
\[
P_t f(x) = \int_{\Rd} f(y) p_t(x,y)\,dy, \quad  f \in C_0(\Rd),\ x \in \Rd, \ t > 0.
\]
} Its infinitesimal generator is
\[
\calL \vphi(x) = \lim_{\epsilon \to 0^+} \int_{|y|>\epsilon} \big( \vphi(x+y)-\vphi(x) \big)\nu(y)\,dy,
\]
for 
$\vphi \in C_c^{\infty}(\Rd)$,
see, e.g., Sato \cite[Theorem 31.5]{Sato99}. 
In particular,
we let
$$
\nu^\alpha(y)
= \mathcal{A}(d,\alpha) |y|^{-d-\alpha}\,,\quad y\in \Rd,
$$
where $\alpha \in (0,2)$ and
\[
\mathcal{A}(d,\alpha) = \frac{ \alpha 2^{\alpha-1}\Gamma\big((d+\alpha)/2\big)}{\pi^{d/2}\Gamma(1-\alpha/2)}.
\]
Then,
\begin{equation}
  \label{eq:trf}
 \psi(\xi)=\psi^\alpha(\xi)= \int_{\Rd} \left(1-\cos(\xi\cdot y)\right)\nu^\alpha(y)dy=|\xi|^\alpha\,,\quad
  \xi\in \Rd\,,
\end{equation}
and 
$\calL = \Delta^{\alpha/2} := - (-\Delta)^{\alpha/2}$, the fractional Laplacian; see, e.g., Kwa\'{s}nicki \cite{MK17}.

Recall that open set $D \subset \Rd$ is {called} Lipschitz if there are numbers $R_0>0$ and $\lambda>0$ such that for every $Q \in \partial D$ there exist an orthonormal coordinate system $\mathrm{CS}_Q$ and a Lipschitz function $f_Q \colon \R^{d-1} \mapsto \R$ with Lipschitz constant $\lambda$ such that if $y=(y_1,\ldots,y_n)$ in $\mathrm{CS}_Q$ coordinates, then
\[
D \cap B(Q,R_0) = \{y \colon y_n > f_Q(y_1,\ldots,y_{n-1})\} \cap B(Q,R_0).
\]
Every Lipschitz set $D$ is $\kappa$-fat, i.e., there exists $\kappa \in (0,1)$ and $R>0$ such that for every $r \in (0,R)$ and $Q \in \overline{D}$ there is a point $A=A_r(Q) \in D \cap B(Q,r)$ such that $B(A,\kappa r) \subset D\cap B(Q,r)$. The pair $(\kappa,R)$ is sometimes called the characteristics of a $\kappa$-fat set $D$. Note that usually $A_r(Q)$ is not uniquely determined.

{As in Introduction,} $\Gamma\in \Rd$ is a Lipschitz cone with vertex at the origin, i.e., (non-empty) open Lipschitz set such that $0 \in \partial \Gamma$ and $ry \in \Gamma$ whenever $y \in \Gamma$ and $r>0$. Note that in this way we exclude $\Gamma=\Rd$ and $\Gamma=\emptyset$ from our discussion (but see a remark following Theorem~\ref{thm:Yaglom}). Without essential loss of generality we may and do assume that $\one := (0,\dots,0,1) \in \Gamma$. We note that if $d=1$, then  necessarily $\Gamma=(0,\infty)$ and our results do not cover $\Gamma=\R\setminus\{0\}$. We also observe that $\Gamma$ is $\kappa$-fat with characteristics $(\kappa,R)$ for every $R>0$, with $\kappa$ independent of~$R$.

The transition density of {the process} killed on exiting $\Gamma$ is defined by Hunt's formula:
\begin{equation}\label{eq:Hunt}
p_t^{\Gamma}(x,y) := p_t(x,y) - \E_x [\tau_{\Gamma}<t;p_{t-\tau_{\Gamma}}(X_{\tau_{\Gamma}},y)], \quad t>0,\, x,y \in \Rd.
\end{equation}
We have $0 \leq p_t^{\Gamma}(x,y) \leq p_t(x,y)$ for all $t>0$ and $x,y\in\Rd$. It is well known that 
$p_t^{\Gamma}$ is symmetric: $p_t^{\Gamma}(x,y)=p_t^{\Gamma}(y,x)$ for all $x,y \in \Rd$ {and $t>0$}. Moreover, the Chapman-Kolmogorov equations (semigroup property) hold {for $p^{\Gamma}$}:
\[
p_{t+s}^{\Gamma}(x,y) = \int_{\Rd}p_t^{\Gamma}(x,z)p_s^{\Gamma} (z,y)\,dz, \quad t,s>0,\,x,y \in \Rd.
\]
For every non-negative or bounded function $f$ we let
\[
P_t^\Gamma f(x) 
= \E_x [\tau_{\Gamma}>t;f(X_t)] = \int_{\Rd} p_t^{\Gamma}(x,y)f(y)\,dy, \quad {x \in \Rd,\ t>0.}
\]
We further note that
\begin{equation}\label{eq:surv}
	\P_x(\tau_{\Gamma}>t) = \int_{\Gamma} p_t^{\Gamma}(x,y)\,dy, \quad {x \in \Rd,\ t>0}.
\end{equation}
Analogous definitions and properties are valid for arbitrary open set $D\subset \Rd$.
Function $u\colon \Rd \mapsto \R$ is called  harmonic with respect to $\bfX$ on open set $D \subset \Rd$ if
\[
u(x) = \E_x u(X_{\tau_B}), \quad x \in B,
\] 
for all open bounded sets $B$ such that $\overline{B} \subset D$.
Here we assume that the integral on the right-hand side is absolutely convergent. The function $u$ is called regular harmonic in $D$ if the identity above is satisfied with $B=D$. The concept of regular harmonicity is related to the notion of harmonic measure $P_D(x,\,\cdot\,)$, i.e., the distribution of $X_{\tau_D}$:
\[
P_D(x,A) = \P_x(X_{\tau_D} \in A), \quad A \in \Rd, \,x \in \Rd.
\]
Namely, function $u$ regular harmonic on $D$ satisfies
\[
u(x) = \int_{D^c} u(y)\,P_D(x,dy), \quad x \in D.
\]
The density of the harmonic measure on $\overline{D}^c$ is called the Poisson kernel and is denoted by $P_D(x,z)$ for $x \in D$, $z \in \overline{D}^c$ (see also \eqref{eq:IW2} below). For simplicity we write $P_{B_r} =: P_r$.

The Green function of $D$ is given by
\[
G_{D}(x,y) = \int_0^{\infty}p_t^D(x,y)\,dt, \quad \,x,y \in \Rd.
\]
In the case of $D=\Rd$ we write $U:=G_{\Rd}$ and $U$ is then called the potential kernel of $\bfX$. If $\psi \in \calR_{\alpha}^0$, for $\alpha\in(0,2)$, and $d\geq 2$, then by the Chung-Fuchs criterion (see, e.g., \cite[Corollary 37.6]{Sato99}),
$\bfX$ is transient and 
the potential kernel is finite $a.e$. The Green function gives rise to the Green operator, defined as
\[
G_D f(x) = \int_{\Gamma}G_D(x,y)f(y)\,dy, \quad x \in \Rd, 
\]
for non-negative or integrable functions $f$. Then, by the Fubini's theorem,
\[
\E_x \int_0^{\tau_D} f(X_t)\,dt = \int_D G_D(x,y)f(y)\,dy.
\]
In particular, by letting $f = \one$ we obtain $G_D \one(x) = \E_x \tau_D$.

Let $r>0$. Consider the \textit{truncated cone}:
\[
\Gamma_r = \Gamma \cap B_r.
\]
The strong Markov property implies that for all $t>0$ and $x,y \in \Rd$,
\[
p_t^{\Gamma}(x,y) = p_t^{\Gamma_r}(x,y) + \E_x\left[\tau_{\Gamma_r}<t;p_{t-\tau_{\Gamma_r}}^{\Gamma}(X_{\tau_{\Gamma_r}},y) \right].
\]
Integrating the identity with respect to $dt$ we obtain
\begin{equation}\label{eq:G_harm}
G_{\Gamma}(x,y) = G_{\Gamma_r}(x,y) + \E_x G_{\Gamma}(X_{\tau_{\Gamma_r}},y), \quad x,y \in \Rd.
\end{equation}
It follows that $x \mapsto G_{\Gamma}(x,y)$ is regular harmonic on $\Gamma_r$ if $|y|\ge r$.

For $x \in D$ the distribution of $(\tau_D,X_{\tau_D-},X_{\tau_D})$ restricted to the event $\{X_{\tau_D-} \neq X_{\tau_D}\}$ has the following density function:
\[
(0,\infty) \times D \times D^c \ni (u,y,z) \mapsto \nu(y,z)p_u^{\Gamma}(x,y), 
\]
that is
for $I \subset (0,\infty)$, $A \in D$ and $B \subset D^c$ we have
\begin{equation}\label{eq:IW}
\P_x (X_{\tau_D-} \neq X_{\tau_D}, \tau_D \in I, X_{\tau_D-} \in A, X_{\tau_D} \in B) = \int_I \int_A\int_{B} p_u^{\Gamma}(x,y)\nu(y,z)\,dz\,dy\,du.
\end{equation}
{This is called the Ikeda-Watanabe formula.} Denote
\[
\kappa_D(y) = \int_{D^c}\nu(y,z)\,dz, \quad y \in D.
\]
We call $\kappa_D$ the killing intensity. Assume now that the set $D$ is Lipschitz. Then, by Sztonyk \cite[Theorem 1]{Sztonyk00}, for all $x \in D$ we have $\P_x(X_{\tau_D} \in \partial D)=0$ and
\begin{equation}\label{e.p1}
\P_x(X_{\tau_{D}-} = X_{\tau_{D}}) = 0.
\end{equation}
Letting $I=(t,\infty)$, $A=D$ and $B=D^c$ in \eqref{eq:IW}, by Chapman-Kolmogorov equations,
\begin{align*}
	\P_x(\tau_{D}>t) &= \int_t^{\infty} \int_{\Gamma} p_u^D(x,y)\kappa_D(y)\,dy\,du \\ &= \int_0^{\infty} \int_{\Gamma} \int_{\Gamma} p_u^D(x,w)p_t^D(w,y)\,dw \kappa_D(y)\,dy\,du \\ &= G_D P_t^D \kappa_D(x), \quad x \in D.
\end{align*} 
Furthermore, $\P_x$-$a.s.$ we have $\tau_D = 0$ for every $x \in D^c$, so
\begin{equation}\label{eq:surv_id}
	\P_x(\tau_D>t) = G_D P_t^D \kappa_D(x), \quad x \in \Rd.
\end{equation}
The following identities, {also} known as Ikeda-Watanabe formulae, are vital for our development. {Namely}, by setting $I=(0,\infty)$ and $A=D$ in \eqref{eq:IW} we obtain 
\begin{equation}\label{eq:IW2}
P_D(x,z) = \int_D G_D(x,y)\nu(y-z)\,dy, \quad x \in D, z \in \overline{D}^c.
\end{equation}
Consequently, for every function $u$ regular harmonic in $D$ with respect to $\bfX$ we have
\begin{equation}\label{eq:Poisson_f}
u(x) = \int_{D^c}P_D(x,z)u(z)\,dz, \quad {x \in D}.
\end{equation}
After Pruitt \cite{Pruitt81}, we define the \textit{concentration functions} for the L\'{e}vy density $\nu$,
\[
K(r) = r^{-2}\int_{B_r} |z|^2 \nu(z)\,dz \qquad \text{and} \qquad h(r) = \int_{\Rd} \bigg(\frac{|z|^2}{r^2} \wedge 1\bigg)\nu(z)\,dz, \quad r>0.
\]
We note that $h$ is strictly decreasing and for all $r>0$ and $\lambda \leq 1$,
\begin{equation}\label{eq:h_doubling}
\lambda^2h(\lambda r) \leq h(r) \leq h(\lambda r).
\end{equation}
{In particular}, $h$ is doubling on $(0,\infty)$. Furthermore, by {Grzywny} \cite[Lemma 4]{TG14}, 
\begin{equation}\label{eq:52}
	\frac{1}{8(1+2d)}h(1/r) \leq \psi^*(r) \leq 2h(1/r),\quad r>0,
\end{equation}
where $\psi^*$ is the {radially non-decreasing} majorant of $\psi$, i.e., 
\[
\psi^*(r) = \sup_{|z|\leq r}\psi(z), \quad {r\geq 0}.
\]
By {Bogdan, Grzywny and Ryznar} \cite[Proposition 2]{KBTGMR-jfa},
\begin{equation}\label{eq:64}
	{\psi(r) \leq} \psi^*(r) \leq \pi^2 \psi(r), \quad {r \geq 0,}
\end{equation}
{where $\psi(r) := \psi(x)$ for $r=|x|$ and $x \in \Rd$.}

We conclude this section by listing collecting some {consequences of} {\bf A1}, {\bf A2} or {\bf A3}. 
First we observe that {\bf A1} implies that for {every} $r_0>0$ there is $c=c(r_0)$ such that
\begin{equation}\label{eq:nu_comp}
        c\nu(r) \leq \nu(r+1) \leq \nu(r), \quad r>r_0.
\end{equation}
Next, note that the monotonicity of the L\'{e}vy density entails that
\begin{equation}\label{eq:53}
	\nu(r) \leq c(d)K(r)r^{-d},\quad r>0.
\end{equation}
\begin{proposition}\label{prop:K}
	Assume {\bf A1}. For every $\lambda \leq 1$ and $r>0$,
	\[
	K(\lambda r) \leq M\lambda^{-\beta}K(r).
	\]
	Furthermore, 
	\[
	\nu(r) \approx r^{-d}K(r), \quad r>0,
	\]
	with the comparability constant depending only on $d$, $M$ and $\beta$. If we additionally assume {\bf A2}, then  for every $R>0$ there exists {constant} $c>0$ such that
	\[h(r)\leq c K(r),\quad r\geq R.\]
\end{proposition}

\begin{proof}
The assumption {\bf A1} together with {Grzywny and Szczypkowski} \cite[Lemma A.3]{TGKS17} immediately imply the first claim. {For large} $R>0$ {the last claim} is a consequence of  {\bf A2}, Potter bounds for $\nu$ and {Grzywny and Szczypkowski} \cite[Lemma 2.5]{TGKS19}. Using positivity and monotonicity of $\nu$ it is easy to make the threshold {$R>0$} arbitrary. 
\end{proof}

\begin{lemma}\label{lem:Ub}
Assume {\bf A1}. If $d\geq 2$ then there is a constant $c=c(d,M,\beta)$ such that
$$U(x)\leq \frac{c}{h(|x|)|x|^d},\quad x\neq 0.$$
For $d=1$ we have
\[G_{(0,\infty)}(x,y)\leq \frac{c}{y\sqrt{h(x)h(y)}},\quad 0<x<y.\]
\end{lemma}
\begin{proof}
First note that due to the Chung-Fuchs criterion the process is transient if $d\geq 2$. By \cite[Lemma 2.2]{TGKS19} and \eqref{eq:52} with \eqref{eq:64} we have, for $s>1$ and $x\in \Rd$,
$$\psi(s x)\leq c M s^\beta \psi(s),$$
where $c$ depends only on $d$. Then the claim is a consequence of {Bogdan, Grzywny and Ryznar} \cite[Lemma 5.6]{KBTGMR-ptrf} and \cite[Theorem 3]{TG14}.
If $d=1$, then the assumption  {\bf A1} implies  global scale invariant Harnack inequality for the process due to {Grzywny and Kwa\'{s}nicki} \cite[Theorem 1.9 and Remark 1.10 e)]{TGMK}. With this in hand one can repeat the proof of Corrolary 5.6 and the upper bound in Corrolary 5.5 in {Grzywny, Le\.{z}aj and Mi\'{s}ta} \cite{TGLLMS21} to get the claim. 
\end{proof}

\begin{proposition}\label{prop:h_psi}
        Assume {\bf A2}. Then $h^{-1}$ has doubling property on $(0,r)$ for every $r>0$. 
        {Furthermore}, for every $r>0$ there is a constant $c>0$  such that
	\[
	\frac{c^{-1}}{h^{-1}(u)} \leq \psi^{-1}(u) \leq \frac{c}{h^{-1}(u)} , \quad u < r.
	\]
\end{proposition}
\begin{proof}
	Observe that the regular variation of $\psi$ together with \eqref{eq:64} entails {\bf B3} in \cite[Lemma 2.5]{TGKS19}. Thus, the first claim follows by a standard extension  argument and monotonicity of $h^{-1}$. 
	 Next observe that by \eqref{eq:52},
	\[
	\frac{1}{h^{-1}(r/2)} \leq \psi^{-1}(r) \leq \frac{1}{h^{-1}(8(1+2d)r)}, \quad r>0.
	\]
Now we may apply the doubling property of $h^{-1}$ to conclude the proof.
\end{proof}

Now we let $u$ be a harmonic function with respect to $\bfX$ in an open set $U$. By the Poisson formula \eqref{eq:Poisson_f},
\[
u(x) = \int_{|z|>r} P_r(0,z)u(x+z)\,dz,
\]
if only $\overline{B(x,r)} \subset U$. Recall that $P_r(x,\,\cdot\,)$ is the Poisson kernel for the ball $B_r$. Then \cite[Lemma 2.2]{TGMK} entails
\[
u(x) \geq \frac{c}{h(r)} \int_{|z|>r} \nu(z)u(x+z)\,dz
\]
for some constant $c \in (0,1]$. Therefore, using \eqref{eq:nu_comp} we conclude that
\begin{equation}\label{eq:weightedL1}
	\int_{\Rd} u(x) \big( 1 \wedge \nu(x) \big) \,dx < \infty.
\end{equation}

Let $D$ be an arbitrary open set. 
\begin{proposition}\label{prop:BHP}
	Assume {\bf A1} and let $x_0 \in \Rd$ and $r>0$. Suppose that non-negative functions $f,g$ are regular harmonic in $D \cap B(x_0,2r)$ and vanish on $D^c \cap B(x_0,2r)$. Then
	\[
	f(x) \stackrel{\cbhit}{\approx} \E_x \tau_{D \cap B(x_0,4r/3)} \int_{\Rd \setminus B(x_0,5r/3)} f(y) \nu(|y-x_0|)\,dy
	\]
	for $x \in D \cap B(x_0,r)$, where $\cbhit=\cbhit (d,M,\beta)$, and 
	\[
	f(x)g(y) \leq \cbhi f(y)g(x), \quad x,y \in D \cap B(x_0,r),
	\]
	with $\cbhi = \cbhit^4$.
\end{proposition}
\begin{proof}
	By \cite[Remark 1.10d]{TGMK} we get that $\cbhit$ depends only on the characteristics of the process $\bfX$. With the notation from \cite{TGMK} we have $R_{\infty}=\infty$, $\alpha=\beta$, and $M$ {in {\bf A1}} is the same {as in \cite{TGMK}}. Therefore, $\cbhit=\cbhit(d,M,\beta)$ {by \cite[Theorem 1.9]{TGMK}}. 
\end{proof}
\begin{proposition}\label{prop:G_anihilation}
	Assume {\bf A1} and let $\phi \in C_c^{\infty}(\Gamma)$. Then $G_{\Gamma} \calL \phi$ is well defined and
	\begin{equation*}
		G_{\Gamma} \calL \phi(x)=-\phi(x), \quad x\in \Rd.
	\end{equation*}
\end{proposition}
\begin{proof}
  Since  $\nu$ is symmetric, for every $y \in \Rd$ we have
	\[
	\big|\calL \phi(y) \big| \leq \norm{\phi}_{C^2(\Rd)} \int_{\Rd} (1 \wedge |z|^2)\nu(z)\,dz < \infty.
      \]
      
	Fix $x \in \Rd$. Choose $R_1$ so that $\supp \phi \subset B_{R_1}$ and set $R=2R_1+|x|$. 
      By Dynkin's formula \cite[(5.8)]{Dynkin65}, for $r>2R$ we have
	\[
	-\phi(x) = \E_x \int_0^{\tau_{\Gamma_r}}\calL \phi(X_t)\, dt = \int_0^{\infty} \E_x \big[\tau_{\Gamma_r}>t;\calL \phi(X_t)\big]\,dt = \int_{\Gamma_r} G_{\Gamma_r}(x,y)\calL \phi(y)\,dy.
	\]
	The application of the Fubini theorem is justified by the facts that $\calL\phi$ is bounded on $\Rd$ and $\E_x \tau_{\Gamma_r}<\infty$ (see, e.g.,  \cite{KBTGMR-ptrf}). We split the integral as follows:
	\begin{align*}
		-\phi(x) &= \int_{\Gamma_R} G_{\Gamma_r}(x,y)\calL \phi(y)\,dy + \int_{\Gamma_r \setminus \Gamma_R} G_{\Gamma_r}(x,y)\calL \phi(y)\,dy \\ &=: I_1(r) + I_2(r).
	\end{align*}
        By Proposition \ref{prop:BHP}, for $y \in \Gamma_R$ and $v \in \Gamma \setminus \Gamma_{2R}$ and some fixed $y_1 \in \Gamma_R$ and $y_2 \in \Gamma \setminus \Gamma_{2R}$,
	\[
	\frac{G_{\Gamma}(v,y)}{G_{\Gamma}(y,y_2)} \leq \cbhi \frac{G_{\Gamma}(v,y_1)}{G_{\Gamma}(y_1,y_2)}.
      \]
      This and \eqref{eq:G_harm} imply that
      \begin{equation}
        \label{eq:30}
        \begin{aligned}
          G_{\Gamma}(x,y)
          &= G_{\Gamma_{2R}}(x,y) + \E_x G_{\Gamma} \big( X_{\tau_{\Gamma_{2R}}},y \big) \\
          &\leq G_{\Gamma_{2R}}(x,y) + c \E_x G_{\Gamma} \big( X_{\tau_{\Gamma_{2R}}},y_1 \big) \cdot \frac{G_{\Gamma}(y,y_2)}{G_{\Gamma}(y_1,y_2)} \\
          &\leq G_{\Gamma_{2R}}(x,y) + c G_{\Gamma}(x,y_1) \cdot \frac{G_{\Gamma}(y,y_2)}{G_{\Gamma}(y_1,y_2)}.
        \end{aligned}
      \end{equation}
      Since $G_{\Gamma}(y, y_2)$ is regular harmonic on $\Gamma_{2R}$ and vanishes on $\Gamma^c$, it is bounded on $\Gamma_R$ by Proposition \ref{prop:BHP}. Therefore, by the boundedness of $\calL \phi$, \eqref{eq:30} and the dominated convergence theorem,
      \begin{equation}
        \label{eq:I1r}
        \lim_{r \to \infty} I_1(r) = \int_{\Gamma_R} G_{\Gamma}(x, y) \calL \phi(y) dy. 
      \end{equation}
 Note that by Proposition \ref{prop:K}, for $y \in \Gamma_R^c$,
   \begin{equation}\label{eq:47}
   	|\calL \phi(y)|\leq \int_{B_{R_1}}|\phi(z)|\nu(y-z)dz  \leq ||\phi||_\infty |B_{R_1}|\nu(|y|-R_1)\lesssim\nu(|y|/2) \lesssim \frac{K(|y|)}{|y|^d}  .
   \end{equation}  
   This and Lemma \ref{lem:Ub} {yield}
\[G_{\Gamma_r}(x,y)|\calL \phi(y)|\lesssim |y|^{-2d}.\]
 Now, the dominated convergence theorem implies that
      \begin{equation}
        \label{eq:I2r}
        \lim_{r \to \infty} I_2(r) =
 \int_{\Gamma \setminus \Gamma_R} G_{\Gamma}(x,y)\calL \phi(y)\,dy. 
      \end{equation}
      Combining \eqref{eq:I1r} and \eqref{eq:I2r} we complete the proof.

    \end{proof}

\section{Heat kernel estimates}\label{s.hke}
We apply to $\Gamma$ standard geometric considerations on $\kappa$-sets, {see} Bogdan, Grzywny and Ryznar \cite[Definition 2]{KBTGMR10} and Chen, Kim and Song \cite[Figure 1]{CKS14}. Namely, for $x \in \Gamma$ and $r>0$ we let  
\[
U^{x,r}=B(x,|x-A_r(x)|+\kappa r/3) \cap \Gamma, \qquad B^{x,r}_1 = B(A_r(x),\kappa r/3),
\]
so that $B_1^{x,r} \subset U^{x,r}$. There is also $A'_r(x)$ and $B^{x,r}_2 = B(A'_r(x),\kappa r/6)$ so that $B(A'_r(x),\kappa r/3) \subset B(A_r(x), \kappa r) \setminus U^{x,r}$ and consequently $\dist(U^{x,r},B_2^{x,r}) \geq \kappa r/6$, see Figure \ref{fig:1}.
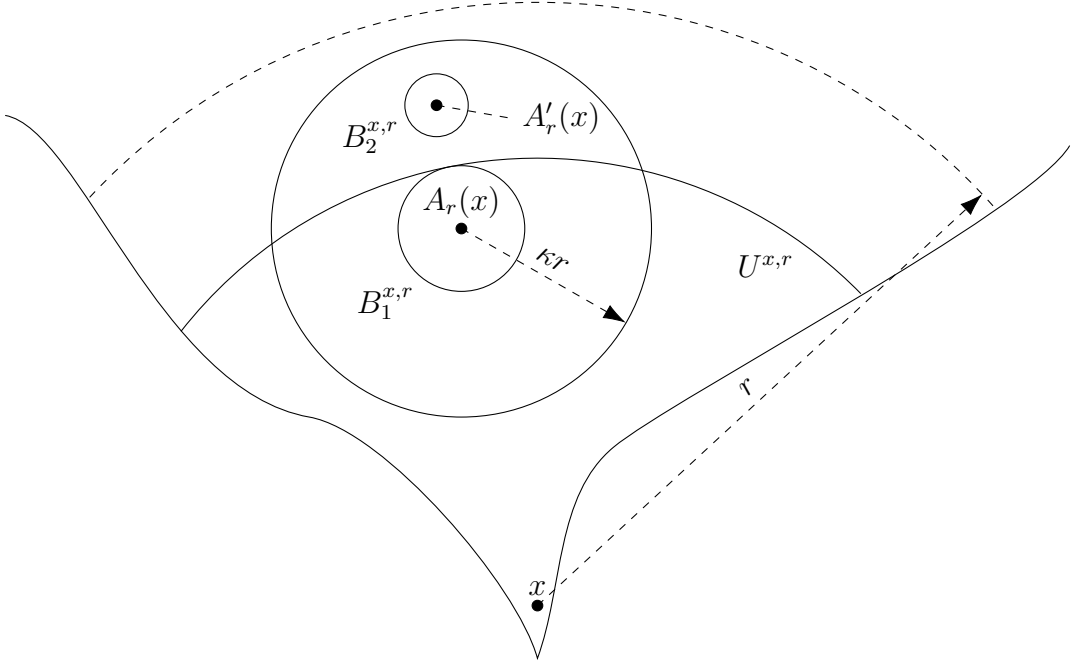
\begin{figure}[H]
	\centering
	\begin{tikzpicture}
		\draw[name path=set] (0,0) .. controls +(350:1) and +(170:2) .. (4,-4)
		.. controls +(350:1) and +(105:1).. ++ (3,-3.2)
		.. controls +(70:1) and +(220:1) .. ++(1,2.8)
		.. controls +(40:1) and +(240:1) .. ++(6,4);
		\filldraw (7,-6.5) circle (2pt) node[above] {$x$} coordinate (X);
		\draw[dashed, name path=arc] ([shift=(41.5:8)]X) arc (41.5:137.5:8);
		\draw[postaction={decorate, decoration={text along path, text={$r${}}, text align={center}, raise = 4pt}}, dashed, -{Latex[length=3mm,width=2mm]}] (X) -- +(43:8cm);
		\filldraw (6,-1.5) circle (2pt) node[above] {$A_r(x)$} coordinate (A);
		\draw (A) circle (2.5);
		\draw[postaction={decorate, decoration={text along path, text={${\kappa}r${}}, text align={center}, raise = 4pt}}, dashed, -{Latex[length=3mm,width=2mm]}] (A) -- +(330:2.5cm);
		\draw (A) circle (2.5/3); 
		\draw ([shift=(44.2:5.1+2.5/3)]X) arc (44.2:142.2:5.1+2.5/3);
		\filldraw (5.673,0.135) circle (2pt) node[above] {} coordinate (A1);
		\draw (A1) circle (2.5/6);
		\draw[dashed] (A1) -- +(350:1cm) node[right] {$A_r'(x)$};
		\node at (10,-2) {${U^{x,r}}$};
		\node at (5,-2.5) {${B_1^{x,r}}$};
		\node at (4.8,-0.3) {${B_2^{x,r}}$};
	\end{tikzpicture}
	\caption{Geometric setting for $\kappa$-fat sets}
	\label{fig:1}
\end{figure}

We next focus on sharp estimates of the heat kernels {$p_t$  and} $p_t^{\Gamma}$.
\begin{proposition}\label{prop:hk_est}
	Assume {\bf A}. There exists $T_1>0$ such that
	\[
	p_t(x) \approx p_t(0) \wedge t\nu(x) \approx \big(\psi^{-1}(1/t)\big)^d \wedge t\nu(x), \quad x \in \Rd,\ t \geq T_1.
	\]
\end{proposition}
\begin{proof}
	Since $X_t$ is symmetric, by \cite[Theorem 5.4]{TGMRBT} and Proposition \ref{prop:K} we get
	\[
	p_t(x)\leq p_t(0)\wedge c(d)t K(|x|)|x|^{-d}\lesssim p_t(0)\wedge t\nu(x), \quad {t>0,\ x \in \Rd}.
	\]
	Next, note that \eqref{eq:64} and the Potter bounds for $\psi$  imply the condition {\bf D3} in \cite[Theorem 3.12]{TGKS19}. This, in view of Proposition \ref{prop:h_psi}, entails the existence of $T_1>0$ such that the claimed upper bound of $p_t(x)$ holds for $t\geq T_1$. 
	
	Furthermore,  \cite[Proposition 5.3]{TGMRBT} gives us
	\begin{equation}\label{eq:t2}p_t(x)\geq c_1 t\nu(x)e^{-c_2 t \psi (1/|x|)}, \quad t>0,\, x\in \Rd.\end{equation}
	By Proposition \ref{prop:h_psi} and \cite[Lemma 2.5]{TGKS19} we have $t\nu(x) \approx \big(\psi^{-1}(1/t)\big)^d$ if $t\psi(1/|x|)= 1$ and $t\geq T_1$. This and the radial monotonicity of $p_t$ imply the lower bound. 
\end{proof}

\begin{proposition}\label{prop:P_tau}
	Assume {\bf A1} and {\bf A2}. For every $T>0$ there is 
	 $c >0$ such that 
	\begin{equation}\label{eq:67}
		\E_z \tau_{B(z,r) \cap \Gamma} \leq c t\P_z(\tau_{\Gamma}>t),\quad z\in \Gamma,\ t\ge T,
	\end{equation}
	where $r=1/\psi^{-1}(1/t)$. Furthermore, 
	\begin{equation}\label{eq:69}
		\P_z (\tau_{\Gamma}>t/2) \approx \P_z (\tau_{\Gamma}>t), \quad z \in \Gamma, \ t\geq T.
	\end{equation}
	
\end{proposition}

\begin{proof}
	Let 
	$r=1/\psi^{-1}(1/t)$. Set $D = B(z,r) \cap \Gamma$ and $A = A_{3r/\kappa}(z)$. 
	First let us consider the case $|z-A| \leq r$. {Then} $D=B(z,r)$, and Pruitt's estimates \cite{Pruitt81} and \eqref{eq:52} {yield},
	\begin{equation}
		\E_z \tau_D  \leq \frac{c}{h(r)}\leq 2c \, t, \quad t >0,
	\end{equation}
	where $c$ depends only on $d$.
 Furthermore, by  \cite[Proposition 5.3]{TGMRBT}
\begin{equation}\label{eq:t1}
 \P_z (\tau_{D}> t)\geq c_1 e^{-c_2 t\psi(1/r)}=c(d),\quad t>0.
\end{equation}
Since $\P_z (\tau_{D}> t) \leq \P_z (\tau_{\Gamma}> t)$, we get
\eqref{eq:67}  in this case.
	
	Now suppose that $|z-A| >r$. Let $V \in \overline{D}^c$. By the Ikeda-Watanabe formula,
	\begin{equation}\label{eq:66}
		\P_x(X_{\tau_{D}}\in V)=\int_{D} G_{D}(x,w)\nu(V-w)dw\geq \inf_{w\in D}\nu(V-w) \E_x\tau_{D}.
	\end{equation}
	The condition $|z-A|>r$ allows for $A' \in B(A,3r)$ such that $B(A',2r) \subset B(A,3r) \setminus D$. Hence, for $V=B(A',r)$ we get, with the aid of Proposition \ref{prop:K} and {\bf A1},
	\[
	\inf_{w\in D}\nu(V-w)\geq \nu(r(2+3/\kappa))|V|\approx K(r)\approx t^{-1},
	\]
	the last comparability resulting from Proposition \ref{prop:K} and \eqref{eq:52}. 
	Thus, by \eqref{eq:66} and \eqref{eq:t1}, 
	\begin{equation}\label{eq:t3}
	\E_z\tau_{D}\lesssim t \P_z(X_{\tau_{D}}\in V) \lesssim t \E_z\left[X_{\tau_{D}}\in V;\P_{X_{\tau_{D}}}(\tau_{B(X_{\tau_{D}},r)}>t)\right]\leq t \P_z(\tau_\Gamma>t),
	\end{equation}
	so \eqref{eq:67} is proved.
	
	We now turn our attention to the proof of \eqref{eq:69}. In the first case $|z-A| \leq r$, by  \eqref{eq:t1},
	\[
	1 \geq \P_z (\tau_{\Gamma}>t/2) \geq \P_z (\tau_{\Gamma}>t) \geq \frac12.
	\]
	Next note that for every $z \in \Gamma$ we have from \cite[Lemma 2.1]{KBTGMR-ptrf}  and \eqref{eq:52},
	\begin{equation}\label{eq:70}
		\P_z(X_{\tau_{D}}\in \Gamma)\leq \P_z\big(\big|X_{\tau_{D}}-z\big|\geq r\big)\leq 24 h(r)\E_z\tau_{D}\leq C(d) t^{-1} \E_z\tau_{D}, \quad t >0.
	\end{equation}
	Therefore, in the case $|z-A|>r$, by Markov inequality and \eqref{eq:70},
	\[
	\P_z(\tau_\Gamma>t/2)\leq\P_z(\tau_{D}>t/2)+\P_z(X_{\tau_{D}}\in \Gamma)\leq c t^{-1}\E_z\tau_{D}.
	\]
	The application of \eqref{eq:67} yields \eqref{eq:69}.
\end{proof}

The next proposition provides lower estimates on the Dirichlet heat kernel.
\begin{proposition}\label{prop:dhk_int1}
	Assume {\bf A1} and {\bf A2}. For every $T>0$ there is $b \geq 1$ and $c>0$ such that for all $x,y \in \Gamma$ and $t\geq T$ satisfying $\delta_{\Gamma}(x) \wedge \delta_{\Gamma}(y) \geq b/\psi^{-1}(1/t)$,
	\[
	p_t^{\Gamma}(x,y) \geq c \big(\psi^{-1}(1/t)\big)^d \quad \mbox{if}\quad |x-y|\leq  b/\psi^{-1}(1/t),
	\]
	 and 
	 \begin{equation}\label{eq:dhk_int2}
	p_t^{\Gamma}(x,y) \geq c t\nu(x-y) \quad \mbox{if}\quad |x-y|\geq  b/\psi^{-1}(1/t).	
	\end{equation}
\end{proposition}
\begin{proof}
	Set $r=1/\psi^{-1}(1/t)$. First we consider $|x-y|\leq br$. The Hunt formula \eqref{eq:Hunt} is
	\[
	p_t^{\Gamma}(x,y) = p_t(y-x) - \E_x [\tau_{\Gamma}<t;p_{t-\tau_{\Gamma}}(y-X_{\tau_{\Gamma}})], \quad x,y \in \Rd.
	\]
	By the radial monotonicity of $p_t$ and \eqref{eq:t2}, for every $t >0$,
	\begin{equation}\label{eq:31}
	p_t(y-x) \geq p_t(b r) \geq c_1 t \nu(b r)e^{-c_2t\psi(\psi^{-1}(1/t)/b)}\gtrsim t \nu(b r), 
	\end{equation}
	with the implied constant depending only on $d$. 
	 Furthermore, by \cite[Theorem 5.4]{TGMRBT}, Proposition \ref{prop:K} and monotonicity of $\nu$,
	\begin{align*}
		\E_x \left[\tau_{\Gamma}<t;p_{t-\tau_{\Gamma}}\big(y-X_{\tau_{\Gamma}}\big)\right] &\lesssim \E_x \left[\tau_{\Gamma}<t;(t-\tau_{\Gamma})\nu\big(y-X_{\tau_{\Gamma}}\big)\right] \\
		 &\leq t\nu\big( \delta_{\Gamma}(y) \big)\P_x(\tau_{\Gamma}<t) \\ &\leq t \nu(b r)  \P_x \big( \tau_{B(x,\delta_{\Gamma}(x))}<t \big)\\
		&\leq  t \nu(b r) \P_x \big( \tau_{B(x, b r)}<t \big).
	\end{align*}
	
	By Potter's bounds for $\psi$, there exists $w>0$ such that 
$$\psi(u)\leq 2 \left(\frac{u}{s}\right)^{3\alpha/4}\psi(s), \quad u\leq s\leq w.$$ 
Set $\lambda=\min\left\{1,w/\psi^{-1}(1/T)\right\}$. The above inequality, Pruitt's estimates \cite{Pruitt81} and  \eqref{eq:h_doubling} together with \eqref{eq:52} imply, for $t \geq T$, 
	\begin{align}\label{eq:68}\begin{aligned}
		\P_0 (\tau_{B_{br}}\leq t)& \leq c_1\,th(br) \leq c_2 t \lambda^{-2}  \psi^* \big(\lambda\psi^{-1}(1/t)/b\big)\\ 
& \leq  2 c_2 t \lambda^{-2}   b^{-3\alpha/4}  \psi^* \big(\lambda \psi^{-1}(1/t)\big)\leq 2 c_2  \lambda^{-2}   b^{-3\alpha/4},\end{aligned}
	\end{align}
where $c_2$ depends only on $d$.
	Thus, by fixing $b$ large enough and putting together \eqref{eq:31} with \eqref{eq:68} we conclude that for $t\geq T$,
	\[
	p_t^{\Gamma}(x,y) \gtrsim t \nu(br)\geq M^{-1} t b^{-d-\beta}\nu(r). 
	\] 
	By  Proposition \ref{prop:K} and \eqref{eq:52}
	we get the first claim.
	
	Now assume that $|x-y|\geq br$.
	By \cite[Lemma 1.10]{KBTGMR-dhk} with $D_1 = B(x,br/2)$ and $D_3 = B(y,br/2)$ we have
	\[
	p_t^{\Gamma}(x,y) \geq t\P_x(\tau_{D_1}>t)\P_y (\tau_{D_3}>t)\inf_{u \in D_1,z \in D_3}\nu(u-z).
	\]
	Observe that by \eqref{eq:h_doubling} and \eqref{eq:52},  
	\[
	th\big(b r/2\big) \approx th\big(r\big)\approx t\psi^* \big(\psi^{-1}(1/t)\lambda\big) \approx 1,
	\]
	with comparability constants depending only on $d$ and $T$.  Thus, in view of \cite[Proposition 5.3]{TGMRBT}, with constant depending only on the dimension {we have the comparison:}
	\[
	\P_x (\tau_{D_1}>t) \approx 1.
	\]
	Similarly, $\P_y(\tau_{D_3}>t) \approx 1$. Moreover, for $u \in D_1$ and $z \in D_3$ we clearly have $|u-z| \leq |x-y|+br \leq 2|x-y|$. From the monotonicity of the L\'{e}vy density and {\bf A1},
	\[
	\inf_{u \in D_1,z \in D_2}\nu(u-z) \geq \nu(2|y-x|) \gtrsim \nu(|y-x|),
	\]
	with comparability constant depending only on $d$, $M$ and $\beta$. The proof is complete.
\end{proof}

\begin{proposition}\label{prop:dhk_int3}
	Assume {\bf A1} and {\bf A2}. For each $T>0$ there is  $c>0$ such that
	\[
	\P_x(\tau_{\Gamma \cap B(x,r)}>t) \geq c\P_x(\tau_{\Gamma}>t), \quad {x \in \Gamma}, \ t \geq T,
	\]
	where $r=1/\psi^{-1}(1/t)$.
\end{proposition}
\begin{proof}
	We follow the proof of \cite[Lemma 1]{KBTGMR10}. 
	Let $A = A_r(x)$, $A' = A'_r(x)$ and denote $D = \Gamma \cap B(x,r)$.
	If $|x-A|\leq \kappa r/2$ then $B(x,\kappa r/2) \subset D$. Since $th(\kappa r/2) \approx 1$ (see the proof of Proposition \ref{prop:dhk_int1}), by \cite[Proposition 5.2]{TGMRBT} we get that $\P_x (\tau_{B(x,\kappa r/2)}>t) \approx 1$.  Hence 
	\[
	 \P_x (\tau_{\Gamma}>t) \approx \P_x (\tau_{D}>t), \quad t>0.
	\]
	We thus assume that $|x-A|>\kappa r/2$. 
	For simplicity we write $U=U^{x,r}$, $B_1=B_1^{x,r}$ and $B_2 = B_2^{x,r}$. By the Ikeda-Watanabe formula  \eqref{eq:IW2} we see that,  for $t\geq T$ and $y\in U$,
	\begin{align}\label{eq:71}\begin{aligned}
		\P_y \big(X_{\tau_{U}} \in B_2\big) &= \int_U G_U(y,w) \nu(B_2-w)\,dw \\ &\approx r^d \nu(r) \E_y \tau_{U} \\ &\approx t^{-1}\E_y \tau_{U},\end{aligned}
	\end{align}
	where the last two comparisons follow from
	Proposition \ref{prop:K}. 
	Next, by Proposition \ref{prop:BHP},
	\[
	\P_x \big( X_{\tau_{U}} \in \Gamma \big) \leq \cbhi \P_A \big( X_{\tau_{U}} \in \Gamma \big) \cdot \frac{\P_x \big(  X_{\tau_{U}} \in B_2 \big)}{\P_A \big(  X_{\tau_{U}} \in B_2 \big)} \leq \cbhi \frac{\P_x \big(  X_{\tau_{U}} \in B_2 \big)}{\P_A \big(  X_{\tau_{U}} \in B_2 \big)}.
	\]
	By Pruitt's estimates we obtain $\E_A \tau_{U}\approx t$. It follows that
	\begin{equation}\label{eq:72}
	\P_x \big( X_{\tau_{U}} \in \Gamma \big) \lesssim \P_x \big(  X_{\tau_{U}} \in B_2 \big).
	\end{equation}
	Thus, by the Markov inequality, \eqref{eq:72} and \eqref{eq:71},
	\begin{equation}\label{eq:73}
		\P_x \big(\tau_{\Gamma}>t\big) \leq \P_x \big( \tau_U > t \big)+ \P_x \big( X_{\tau_U} \in \Gamma \big) \lesssim t^{-1}\E_x \tau_U.
	\end{equation}
	Moreover, by the verbatim repetition of the argument from \cite[(4.8)--(4.9)]{CKS14} we get
	\[
	\P \big( \tau_{B_{\kappa r/6}}>t \big)\P_x \big( X_{\tau_U} \in B_2 \big) \leq \P_x ( \tau_V>t )
	\]
	with $V = B(x,|x-A|+\kappa r) \cap \Gamma$. Therefore, combining \cite[Proposition 5.2]{TGMRBT}, \eqref{eq:71} and \eqref{eq:73} we obtain the claim.
\end{proof}
\begin{lemma}\label{lem:dhk_est}
	Assume {\bf A}. Then there exists $T_2>t_0$ such that
	\[
	p_t^{\Gamma}(x,y) \approx \P_x(\tau_{\Gamma}>t)\P_y(\tau_{\Gamma}>t)p_t(y-x), \quad x,y \in \Rd,\ t \geq T_2,
	\]
\end{lemma}
\begin{proof}
	First we {focus} on the upper bound. Observe that by the semigroup property,
	\[
	p_t^\Gamma(x,y)\leq p_{t/2}(0)\P_x(\tau_{\Gamma}>t/2),\quad x,y\in\Gamma, \, t>0.
	\]
	Thus, by Potter bounds for $\psi^{-1}$ together with Propositions \ref{prop:hk_est} and \ref{prop:P_tau}, for $t\geq 2T_1$, 
	\[
	p_t^\Gamma(x,y)\leq p_{t}(0)\P_x(\tau_{\Gamma}>t),\quad x,y\in\Gamma.
	\]
	Set $r=1/(4\psi^{-1}(1/t))$. Let us consider the case $t\psi^* \big(1/|x-y|\big)\leq 1$. Note that this implies that $|x-y| \geq 4r$. Let $D_1=B(x,r)\cap \Gamma$, $D_3=\Gamma\setminus B(x,|x-y|/2)$ and $D_2=\Gamma\setminus(D_1\cup D_3)$. Then by radial monotonicity, \cite[Theorem 5.4]{TGMRBT}, Proposition \ref{prop:K} and {\bf A1},
	\begin{equation}\label{eq:54}
		\sup_{s< t,z \in D_2}p_s(z,y)\leq \sup_{s\leq t}p_s(|x-y|/2)\leq C(d) t \nu(|x-y|).
	\end{equation}
	Moreover, by {\bf A1} we get
	\begin{equation}\label{eq:55}
		\sup_{z\in D_1,w\in D_3}\nu(z-w)\leq \nu(|x-y|/4)\leq C(d,M,\beta) \nu(|x-y|).
	\end{equation}
	Thus, using \cite[Lemma 1.10]{KBTGMR-dhk}, \eqref{eq:54} and \eqref{eq:55} we obtain
	\[
	p_t^\Gamma(x,y)\leq c(d,M,\beta)\left(t\P_x(X_{\tau_{D_1}}\in D_2)+\E_x\tau_{D_1}\right)\nu(|x-y|).
	\]
	Therefore, by \eqref{eq:70} and Proposition \ref{prop:P_tau} we conclude that for $t\psi^* \big(1/|x-y|\big) \leq 1$,
	\[
	p_t^{\Gamma}(x,y) \leq c(d,M,\beta)\P_x(\tau_{\Gamma}>t)\nu(|x-y|).
	\]
	By Proposition \ref{prop:K}, \eqref{eq:52} and \cite[Remark 2]{WCTGBT}, $t\nu(|x-y|) \approx p_t(0)$ for $ t\psi^*(1/|x-y|) \approx 1$, 
	so {Proposition \ref{prop:hk_est} entails}
	\[
	p_t^{\Gamma}(x,y) \leq c(d,M,\beta)\P_x(\tau_{\Gamma}>t)p_t(x-y), \quad x,y \in \Gamma, \ t \geq 2 T_1.
	\]
	Applying symmetry of $p_t^{\Gamma}$, the semigroup property and Proposition \ref{prop:P_tau}, we arrive at the desired upper bound.
	
	Now we turn to the lower bound. This part of proof is inspired by the proofs of \cite[Lemma 5]{KBTGMR10} and \cite[Theorem 1.3]{CKS14}. Let $b$ be taken from Proposition \ref{prop:dhk_int1} and set $r_1 = 6b r/\kappa$ with $r=1/\psi^{-1}(1/t)$. By the semigroup property,
	\begin{align*}
		p_t^{\Gamma}(x,y) &= \int_{\Gamma \times \Gamma} p_{t/3}^{\Gamma}(x,u) p_{t/3}^{\Gamma}(u,v) p_{t/3}^{\Gamma}(v,y)\,du\,dv \\ &\geq \int_{B_2^{x,r_1} \times B_2^{y,r_1}} p_{t/3}^{\Gamma}(x,u) p_{t/3}^{\Gamma}(u,v) p_{t/3}^{\Gamma}(v,y)\,du\,dv \\ &\geq \inf_{u \in B_2^{x,r_1},v\in B_2^{y,r_1}} p_{t/3}^{\Gamma}(u,v) \int_{B_2^{x,r_1}} p_{t/3}^{\Gamma}(x,u)\,du \int_{B_2^{y,r_1}} p_{t/3}^{\Gamma}(y,v)\,dv.
	\end{align*}
	Note that by the definition of $r_1$ we have $\delta_{\Gamma}(u)\wedge \delta_{\Gamma}(v) \geq b/\psi^{-1}(1/t)$. Thus, by Proposition \ref{prop:dhk_int1} and \ref{prop:hk_est} 
	 there is $c>0$ {such that} 	for $t \geq T_1$,
	\begin{equation}\label{eq:77}
	\inf_{u \in B_2^{x,r_1},v \in B_2^{y,r_1}} p_{t/3}^{\Gamma}(u,v) \geq c\left(\big(\psi^{-1}(1/t)\big)^d \wedge t\nu(y-x)\right) \geq cp_t(y-x), \quad {x,y \in \Gamma}.
	\end{equation}
 By \cite[Lemma 1.10]{KBTGMR-dhk}, for $x\in \Gamma$, $u\in B_2^{x,r}$, $D_1 = U^{x,r_1}$ and $D_3 = B(A'_{r_1}(x),\kappa r_1/4)$ we get
	\[
	p_{t/3}^{\Gamma}(x,u) \geq ct\P_x(\tau_{D_1}>t/3) \P_u (\tau_{D_3}>t/3)\inf_{z \in D_1, w \in D_3
	}\nu(w-z).
	\]
	{By the radial} monotonicity of $\nu$, {\bf A1}, Proposition \ref{prop:K} and \eqref{eq:52}, 
	\begin{align}\label{eq:75}\begin{aligned}
	t\inf_{z \in D_1, w \in D_3 
	}\nu(w-z) &\geq t\nu(br/2) \\ &\geq c(d,M,\beta)t\nu\big( 1/\psi^{-1}(1/t) \big) \\ &\geq c\big(\psi^{-1}(1/t)\big)^d, \quad t>T_1.\end{aligned}
	\end{align}
	Next, {we observe that}	by the same argument as in the proof of  \eqref{eq:t1},
	\begin{equation}\label{eq:76}
	\P_u (\tau_{D_3}>t) \geq \P_u \big( \tau_{B(u,\kappa r_1/12)}>t \big) \gtrsim 1,\quad t>0, \ u\in B_2^{x,r}.
	\end{equation}
	Furthermore, since for $r_2=1/\psi^{-1}(3/t)$ we have $r_1 \geq r \geq r_2$, by Proposition \ref{prop:dhk_int3}, 
	\begin{align}\label{eq:74}\begin{aligned}
		\P_x(\tau_{D_1}>t/3) &=\P_x(\tau_{D \cap B(x,r_1)}>t/3) \\ &\geq \P_x(\tau_{D \cap B(x,r_2)}>t/3) \\ & \geq c\P_x(\tau_{\Gamma}>t/3) \\ &\geq c\P_x(\tau_D>t) \end{aligned}
	\end{align}
	for $t \geq T_1$, with $c>0$. Thus, combining \eqref{eq:75}, \eqref{eq:76} and \eqref{eq:74} we conclude that
	\begin{align*}
		\int_{B_2^{x,r_1}}p_{t/3}^{\Gamma}(x,u)\,du \geq c\big|B_2^{x,r_1}\big|\big(\psi^{-1}(1/t)\big)^d\P_x(\tau_D>t) = c\P_x(\tau_D>t)
	\end{align*}
	for $t \geq T_1$. In the same spirit we prove that
	\[
	\int_{B_2^{y,r_1}}p_{t/3}^{\Gamma}(y,v)\,dv \geq c\P_x(\tau_D>t), \quad t \geq T_1.
	\]
	Putting this together with \eqref{eq:77} completes the proof.
\end{proof}

\section{Rescaled process}\label{s.rp}
Let $s \geq 1$.  As aforementioned in  Introduction, we define the rescaled process $\bfX^s = \{X_t^s \colon t\geq0\}$ by setting
\[
X_t^s = \psi^{-1}(1/s)X_{st}.
\]
For the sake of clarity and consistency of notation, every object corresponding to the rescaled process $\bfX$ will be {marked} by a superscript s. For instance, we write $\psi^s$ for the characteristic exponent of $X_t^s$. Observe that for {every} $\xi \in \Rd$,
\begin{align*}
	\E e^{i \xi \cdot X_t^s} = \E e^{i \xi \cdot \psi^{-1}(1/s)X_{st}} =
	 e^{-st\psi \left(\psi^{-1}(1/s)\xi\right)}.
\end{align*}
Thus,
\begin{equation}\label{eq:psis}
\psi^s(\xi) = s \psi \big( \psi^{-1}(1/s)\xi \big), \quad \xi \in \Rd.
\end{equation}
The reader may also easily verify that
\begin{equation}\label{eq:nus}
\nu^s(x) = \frac{s}{\big(\psi^{-1}(1/s)\big)^d} \nu \bigg( \frac{x}{\psi^{-1}(1/s)} \bigg), \quad x \in \Rd.
\end{equation}
Let us collect some basic properties of $\bfX^s$. First, observe that {\bf A1} holds for $\nu^s$ with the same constant $M$ (this simple observation will have profound implications).  Next, we note that under {\bf A2} the tail of the L\'{e}vy measure $t \mapsto \nu (\{x\colon |x|>t\})$ is in $\calR_{-\alpha}^{\infty}$. Thus, by \cite[Proposition 3.1]{TGLLMS21}, we get that $\psi \in \calR_{\alpha}^{0}$. Then it follows that for every $\xi \in \Rd$,
\begin{equation}\label{eq:psi_reg}
\lim_{s \to \infty} \psi^s(\xi) = \frac{\psi \big(\psi^{-1}(1/s)\xi\big)}{\psi \big(\psi^{-1}(1/s)\big)} = |\xi|^{\alpha}.
\end{equation}
Furthermore, under {\bf A2} for every {$x \in \Rd \setminus \{0\}$} we have 
\begin{equation}\label{eq:nu_conv}
	\lim_{s \to \infty} \nu^s(x) = c_{d,\alpha} |x|^{-d-\alpha}.
\end{equation}
{Indeed,}
\[
\nu^s(x) = \frac{s}{\big( \psi^{-1}(1/s) \big)^d} \nu \bigg( \frac{x}{\psi^{-1}(1/s)} \bigg) = |x|^{-d} \frac{\nu \bigg( \frac{x}{\psi^{-1}(1/s)} \bigg)}{\bigg( \frac{|x|}{\psi^{-1}(1/s)} \bigg)^{-d} \psi \bigg( \frac{\psi^{-1}(1/s)}{x} \bigg)} \cdot \frac{\psi \bigg( \frac{\psi^{-1}(1/s)}{x} \bigg)}{\psi \big( \psi^{-1}(1/s) \big)}.
\]
Therefore, the claim follows by \cite[Theorem 7(iii)]{WCTGBT} and the fact that $\psi \in \mathcal{R}_{\alpha}^0$.
\begin{proposition}\label{prop:nu_h_conv}
	Assume {\bf A2}. Then for every $x \in \Gamma$,
	\[
	\lim_{s \to \infty}\kappa_{\Gamma}^s(x) = \kappa_{\Gamma}^{\alpha}(x).
	\]
	Furthermore, if we additionally assume {\bf A1}, then for every $r>0$,
	\[
	\lim_{s \to \infty} h^s(r) = h^{\alpha}(r).
	\]
\end{proposition}
\begin{proof}
	Fix $x \in \Gamma$ and $r>0$. Recall that
	\[
	\kappa_{\Gamma}^s(x) = \int_{\Gamma^c}\nu^s(y-x)\,dy.
	\]
	Since $\nu \in \calR_{-d-\alpha}^{\infty}$, by Potter's bounds 
	there is $r>0$ such that for all $|z_1|\geq |z_2| \geq r$,
	\[
	\frac{\nu(z_1)}{\nu(z_2)} \leq 2\bigg(\frac{|z_1|}{|z_2|}\bigg)^{-d-\alpha/2}.
	\]
	Therefore, due to definition of $\nu^s$ \eqref{eq:nus} there is $s_0 \geq 1$ such that for all $s \geq s_0$,
	\[
	\frac{\nu^s(z_1)}{\nu^s(z_2)} \leq 2\bigg(\frac{|z_1|}{|z_2|}\bigg)^{-d-\alpha/2}
	\]
	for all $|z_1| \geq |z_2| \geq \delta_{\Gamma}(x)$.
	Therefore, for $s\geq s_0$,
	\[
	\nu^s(y-x) \leq 2 \delta_{\Gamma}(x)^{d+\alpha/2} \nu^s(\delta_{\Gamma}(x)) |y-x|^{-d-\alpha/2}, \quad y \in \Gamma^c.
	\]
	Moreover, by \eqref{eq:nu_conv} we get that $\nu^s(\delta_{\Gamma}(x)) \leq 2\nu^{\alpha}(\delta_{\Gamma}(x))$ for $s$ large enough, {so an} application of the dominated convergence theorem finishes the proof of the first part.
	
	Next, observe that for {every} $s \geq 1$,
	\[
	h^s(r) = K^s(r) + \int_{|x|>r}\nu^s(z)\,dz, \quad r>0.
	\]
	The convergence of the second component is established in the same way as $\kappa_{\Gamma}^s$. Furthermore, by {\bf A1}, 
	\[
	|y|^2\nu^s(y) \leq M\nu^s(r) r^{d+\beta} |y|^{-d-\beta+2}, \quad |y| \leq r.
	\]
	Since $\nu^s(r) \leq 2\nu^\alpha(r)$ due to \eqref{eq:nu_conv} for $s$ large enough, another application of the dominated convergence theorem finishes the proof.
\end{proof}
We define
\[
\psi^{-1,s}(u) = \sup \{ r>0 \colon \psi^{*,s}(r)=u \},
\]
where $\psi^{*,s}(r) = \sup_{|z| \leq r}\psi^s(z)$. Then,
\begin{equation}\label{eq:36}
	\lim_{s \to \infty} \psi^{-1,s}(u) = u^{1/\alpha}.
\end{equation}
Indeed,
\[
\psi^{*,s}(r)  = s \sup_{|z|\leq r} \psi \big( \psi^{-1}(1/s)z \big) = s \psi^{*,s}\big( \psi^{-1}(1/s)r \big),
\]
therefore,
\begin{equation}\label{eq:psi_invs}
\psi^{-1,s}(u)  = \sup \{ r>0\colon s\psi^*\big( \psi^{-1}(1/s)r \big) = u \} = \frac{\psi^{-1}(u/s)}{\psi^{-1}(1/s)}.
\end{equation}
Since  $\psi^{-1} \in \calR_{1/\alpha}^0$ by \cite[the proof of Theorem 1.5.12]{RegularVariation89}, we obtain \eqref{eq:36}. 

Let us derive a formula for the heat kernel $p_t^s$. We have
\begin{align*}
	\int_{\Rd} e^{i \xi \cdot x}p_t^s(x)\,dx &= e^{-t\psi^s(\xi)} =e^{-ts\psi \left( \psi^{-1}(1/s)\xi \right)} = \int_{\Rd} e^{i \psi^{-1}(1/s)\xi \cdot x}p_{ts}(x)\,dx.
\end{align*}
By a change of variables we get
\[
\int_{\Rd} e^{i \xi \cdot x}p_t^s(x)\,dx = \frac{1}{\left(\psi^{-1}(1/s)\right)^d}\int_{\Rd} e^{i\xi \cdot u} p_{ts} \left( \frac{u}{\psi^{-1}(1/s)} \right)\,du,
\]
thus,
\begin{equation}\label{eq:hk}
	p_t^s(x) = \frac{1}{\left(\psi^{-1}(1/s)\right)^d} p_{ts} \left( \frac{x}{\psi^{-1}(1/s)} \right).
\end{equation}

Our {next} goal is to prove that the heat kernel of the rescaled process converges to the heat kernel of the limiting $\alpha$-stable L\'{e}vy process.

\subsection{Convergence of the heat and potential kernels}
With the tools from the previous subsection at hand, we are now able to prove the convergence of heat kernels. In fact, we will prove that the same holds true for the Dirichlet heat kernel {of $\Gamma$} and the Green function {of $\Gamma$}, {too}.
\begin{lemma}\label{lem:hk_conv}
	Assume {\bf A2} and {\bf A3}. Then, 
	for every $t>0$ and $x \in \Rd$,
	\begin{equation}\label{eq:hk_conv}
		\lim_{s \to \infty} p_{t}^{s}(x) = p^{\alpha}_{t} (x).
	\end{equation}
\end{lemma}

\begin{proof}
	Fix $t>0$. By a change of variables, we get that $e^{-t\psi^s} \in L^1(\Rd)$ for $s>t_0/t$. Thus, by the Fourier inversion formula,
	\begin{equation*}
		p^{s}_{t} (x)  =  \frac{1}{(2 \pi)^{d}} \int_{\mathbb{R}^{d}} e^{-i  \xi \cdot x } e^{-t \psi^{s}(\xi)} d\xi, \quad {x \in \Rd},
	\end{equation*}
	if $s$ is large enough.	For any given $\delta, s > 0$ we split {the above} integral {according to}:
	\begin{equation*}
		\mathbb{R}^{d} = \left\{ \xi \colon \psi^{-1}( 1/s) | \xi | < \delta \right\} \cup \left\{ \xi \colon \psi^{-1}( 1/s) | \xi | \geq \delta \right\} = D_1(\delta,s) \cup D_2(\delta,s).
	\end{equation*}
	Let us first consider {$D_1(\delta,s)$}. Since $\psi \in \calR_0^{\alpha}$, by Potter's bounds 
	there exists $\delta_1>0$ such that, for all $0 < |x|, |y| \leq \delta_{1}$,
	\begin{equation}\label{eq:59}
		\psi(y) \geq \frac{1}{2} \psi(x) \left(\frac{|y|}{|x|} \right)^{\alpha} \min \left\{ \frac{|y|}{|x|}, \frac{|x|}{|y|} \right\}^{\alpha/2}.
	\end{equation}
	 If $s \geq 1/\psi(\delta_{1})$, then we can let $x = \psi^{-1}(1/s)\one$ and $y = \psi^{-1} (1/s) \xi$ with $\xi \in D_1(\delta_1,s)$, so
	\begin{equation}
		t s \psi\big( \psi^{-1} (1/s) \xi \big) \geq \frac{1}{2} t | \xi |^{\alpha} \min \left\{ |\xi|, |\xi|^{-1} \right\}^{\alpha/2}. \nonumber
	\end{equation}
	Thus,
	\begin{equation}\label{eq:57}
		e^{- t s \psi( \psi^{-1} (1/s) \xi )} \leq e^{- \frac{1}{2} t | \xi |^{\alpha} \min \left\{ |\xi|, |\xi|^{-1} \right\}^{\alpha/2}},
	\end{equation}
	on $D_1 (\delta_1,s)$, provided that $s \geq 1/\psi(\delta_{1})$.
	
	Next, we turn our attention to the second integral for $\delta=\delta_1$. Using the substitution $u = \psi^{-1} \left( 1/s \right) \xi$ and applying \eqref{eq:64} we obtain
	\begin{align*}
		\left| \int_{\psi^{-1} \left( 1/s \right) | \xi | \geq \delta} e^{-i \xi \cdot x} e^{-t s \psi \left( \psi^{-1} \left( 1/s \right) \xi \right)} d \xi \right| & \leq  \int_{| u | \geq \delta} \frac{e^{-t s \psi \left( u \right)}}{\left(\psi^{-1} \left( 1/s \right) \right)^{d}} du\\& \leq  \left(\psi^{-1} \left( 1/s \right) \right)^{-d}\int_{| u | \geq \delta} e^{-t s \psi^* \left( u \right)/\pi^2} du\\
		&\leq \left(\psi^{-1} \left( 1/s \right) \right)^{-d}e^{-(s t-t_0) \psi^* \left( \delta \right)/\pi^2}\int_{| u | \geq \delta} e^{-t_0 \psi^* \left( u \right)/\pi^2} du.
		\end{align*}
Recall that the assumption $\psi \in \mathcal{R}^{0}_{\alpha}$ implies that $\psi^{-1} \in \calR^0_{1/\alpha}$. It follows by {\bf A3} that the above integral tends to $0$ when $s$ goes to infinity. This, the dominated convergence theorem, \eqref{eq:psi_reg} and the Fourier inversion formula yield
	\begin{equation*}
		\lim_{s \to \infty} p^{s}_{t} (x) = p^{\alpha}_{t}(x),
	\end{equation*}
	as claimed.	
\end{proof}

Now we turn our attention to uniform estimates for the heat kernel of $\bfX^s$.
	
\begin{lemma}\label{lem:hks_est}
	Assume {\bf A}. For all $s \geq 1$ we have
	\[
	p_t^s(x)\approx p_t^s(0)\wedge t\nu^s(x)\approx \big(\psi^{-1,s}(1/t)\big)^{d}\wedge t\nu^s(x), \quad x\in\R^d, t\geq T_1/s,
	\]
	with the same comparability constant as in Proposition \ref{prop:hk_est}.
\end{lemma}
\begin{proof}
	First, observe that the case $s=1$ is the contents of Proposition \ref{prop:hk_est}. We will extend the result to arbitrary $s \geq 1$. Observe that by \eqref{eq:nus}, \eqref{eq:psi_invs}, \eqref{eq:hk} and Proposition \ref{prop:hk_est} we have 
	\begin{align*}
		p_t^s(x) &= \frac{1}{\left(\psi^{-1}(1/s)\right)^d} p_{st} \left( \frac{x}{\psi^{-1}(1/s)}\right) \\ &\approx \frac{1}{\left(\psi^{-1}(1/s)\right)^d} \left( \big(\psi^{-1}(1/(st))\big)^d \wedge ts\nu \bigg( \frac{s}{\psi^{-1}(1/s)} \bigg) \right) \\ &= \psi^{-1,s}(1/t) \wedge t\nu^s(x), \quad t \geq T_1/s,
	\end{align*}
	with the implied constant independent of $s$. The proof is completed.
\end{proof}
We next deal with the Dirichlet heat kernel. Since $\Gamma$ is a cone,
\begin{align}\label{eq:taus}
\begin{aligned}\tau_{\Gamma}^s &= 
 \inf \left\{t>0 \colon  \psi^{-1}(1/s)X_{ts} \notin \Gamma\right\}
 = \frac{\tau_{\Gamma}}{s},\end{aligned}
\end{align}
so
\begin{equation}\label{eq:survs}
	\P_x(\tau_{\Gamma}^s>t) = \P_{x/\psi^{-1}(1/s)}(\tau_{\Gamma}>st).
\end{equation}
Here and below we write $\tau_{\Gamma}^s = \tau_{\Gamma}(\bfX^s)$ to point out that the functional $\tau_{\Gamma}$ is applied to the rescaled process. 
By \eqref{eq:hk} and \eqref{eq:taus} and the Hunt formula,
\begin{align*}
	p_t^{\Gamma,s}(x,y) &= p_t^s(x,y) - \E_x \left[ \tau_t^s<t; p_{t-\tau_{\Gamma}^s}^s\left( X^s_{\tau_{\Gamma}^s},y \right) \right] \\ &= \frac{1}{\left(\psi^{-1}(1/s)\right)^d}p_{ts} \left(\frac{y-x}{\psi^{-1}(1/s)}\right) \\ &- \E_{x/\psi^{-1}(1/s)} \left[ \tau_{\Gamma}<st; \frac{1}{\left(\psi^{-1}(1/s)\right)^d} p_{st-\tau_{\Gamma}} \left( \frac{\psi^{-1}(1/s)X_{\tau_{\Gamma}}-y}{\psi^{-1}(1/s)} \right) \right] \\ &= \frac{1}{\left(\psi^{-1}(1/s)\right)^d} \left( p_{ts} \left(\frac{y-x}{\psi^{-1}(1/s)}\right) - \E_{x/\psi^{-1}(1/s)} \left[ \tau_{\Gamma}<st;  p_{st-\tau_{\Gamma}} \left( X_{\tau_{\Gamma}} - \frac{y}{\psi^{-1}(1/s)} \right) \right]  \right),
\end{align*}
therefore,
\begin{equation}\label{eq:dhks}
	p_t^{\Gamma,s}(x,y) = \frac{1}{\left(\psi^{-1}(1/s)\right)^d} \ p_{st}^{\Gamma} \left( \frac{x}{\psi^{-1}(1/s)},\frac{y}{\psi^{-1}(1/s)} \right).
\end{equation}

\begin{lemma}
	\label{prop:dhk}
	Assume {\bf A2} and {\bf A3}. Then for every $t > 0$ and $x$, $y \in \Gamma$, we have 
	\[
	\lim_{s \to \infty}
	p^{\Gamma,\, s}_t(x,\, y) = p^{\Gamma,\, \alpha}_t(x,\, y).
	\]
\end{lemma}

\begin{proof}
	Fix $t>0$ and $x,y \in \Gamma$. By \cite[Corollary 7]{KBTGMR-jfa}, \eqref{eq:64} and \eqref{eq:psis}, we have that
	\[   
	p_t^s(u) 
	\leq C\frac{s t \psi \left( \psi^{-1}( 1/s ) / |u| \right)}{|u|^d}.
	\]
	Let $\delta=\dist(y,\Gamma^c)$. Since $\psi \in \calR_{\alpha}^0$, from Potter's theorem 
	there is $s_0$ such that 
	\[ s \psi \left( \psi^{-1}(1/s) / |u| \right) \leq 2|u|^{- \alpha/2
		}, \quad 
	|u| \geq \delta,
	\]
	if only $s>s_0$. Therefore	for $s$ large enough,
	\begin{equation}
		\label{e:pts}
		p_t^s (u) \leq \frac{C t }{|u|^{d+\alpha/2}}, \quad |u| \geq \delta.
	\end{equation}
	Next, note that $\left| X_{\tau_{\Gamma}} - y \right| \geq \dist(y,\, \Gamma^c)$. It follows that
	\begin{equation}\label{eq:60}
	\E_x \left[ t- \tau_{\Gamma} \leq \epsilon;p_{t - \tau_{\Gamma}}^s(X_{\tau_{\Gamma}}, y) \right]
	\end{equation}
	is uniformly small provided $\epsilon$ is
	small enough. We have also from \eqref{e:pts} that
	\begin{equation}
		\label{e:largeR}
		\E_x^s \left[ |X_{\tau_{\Gamma}}| > R;p_{t - \tau_{\Gamma}}^s(X_{\tau_{\Gamma}},\, y)	 \right] \to 0 \quad \text{as\ } R \to \infty 
	\end{equation}
	uniformly in $s$ large enough. 	Recall that by the proof of Lemma \ref{lem:hk_conv},
	
	\[ 
	p_t^s(x) = \frac{1}{(2 \pi)^d} \int_{\R^d} e^{-i \xi \cdot x } e^{-t \psi^s(\xi)} \, d \xi 
	\]
	for $s$ large enough. Then,
	\[ 
	\left| \frac{\partial}{\partial t} p_t^s(x) \right| + \left| \nabla p_t^s(x) \right| \leq \frac{1}{(2 \pi)^d} \int_{\R^d} e^{-t \psi^s(\xi)} \left( \psi^s(\xi) + |x| |\xi| \right) \,d \xi, 
	\]
	is uniformly bounded for $\epsilon \leq t \leq T$ and $|x| \leq R$ by the
	same arguments as in the proof of Lemma \ref{lem:hk_conv}. Therefore the functions $\{p_t^s(\,\cdot\,)\colon s \geq s_0\}$ are
	equicontinuous for $\epsilon \leq t \leq T$, $|x| \leq R$ and some $s_0 \geq 1$, hence
	\[ \lim_{s \to \infty} p_t^s (x) = p_t^{\alpha}(x) \]
	uniformly in $\epsilon \leq t \leq T$ and $|x| \leq R$. 
	This and \eqref{e:largeR} imply that
	\begin{equation}\label{eq:61} 
	\E_x \left[ \tau^s_{\Gamma} < t - \epsilon;p_{t - \tau_{\Gamma}}^s \left( X^s_{\tau^s_{\Gamma}},\, y \right) \right] - \E_x \left[ \tau^s_{\Gamma} < t - \epsilon;p_{t -
		\tau^s_{\Gamma}}^{\alpha} \left( X^s_{\tau^s_{\Gamma}},\, y \right) \right] \to 0 
	\end{equation}
	as $s \to \infty$. Now, in view of \eqref{eq:psi_reg} and 
	 Jacod and Shiryaev \cite[Corollary VII.3.6]{JS87}, the distribution of $\bfX^s$ converges weakly in $\calD_{[0,T]}$ to the distribution of $\bfX^{\alpha}$. Thus, by Proposition~\ref{prop:continuity} and Corollary~\ref{cor:boundary} we have for every $\epsilon > 0$ that 
	\begin{equation}\label{eq:62} 
	\E_x \left[ \tau^s_{\Gamma} < t - \epsilon;p_{t -
		\tau^s_{\Gamma}}^{\alpha} \left( X^s_{\tau^s_{\Gamma}}, y \right) \right]
	\to \E_x \left[\tau^\alpha_{\Gamma} < t - \epsilon; p_{t -
		\tau^\alpha_{\Gamma}}^{\alpha} \left( X^\alpha_{\tau^\alpha_{\Gamma}}, y \right) \right]
	\end{equation}
	as $s \to \infty$. Therefore, combining \eqref{eq:60}, \eqref{e:largeR}, \eqref{eq:61} and \eqref{eq:62} we arrive at
	\[ 
	\E_x \left[ \tau^s_{\Gamma} < t;p_{t - \tau^s_{\Gamma}}^s \left( X^s_{\tau^s_{\Gamma}}, y \right) \right]	\to	\E_x \left[ \tau^\alpha_{\Gamma} < t;p_{t - \tau^\alpha_{\Gamma}}^\alpha \left(	X^\alpha_{\tau^\alpha_{\Gamma}}, y \right) \right],
	\]
	and by the Hunt formula \eqref{eq:Hunt} with Lemma \ref{lem:hk_conv},
	\[ 
	\lim_{s \to \infty} p_t^{\Gamma,\, s}(x, y) = p_t^{\Gamma,\, \alpha}(x,y).
	\]
\end{proof}
\begin{lemma}\label{lem:dhks_est}
	Assume {\bf A}. For all $s \geq 1$ we have
	\[
	p_t^{\Gamma,s}(x,y) \approx \P_x(\tau^s_\Gamma>t)\P_y(\tau^s_\Gamma>t)p^s_t(y-x),\quad x,y\in \Gamma, t\geq T_2/s,
	\]
	where the comparability constant is taken from Lemma \ref{lem:dhk_est}.
\end{lemma}
\begin{proof}
	The case $s=1$ follows from Lemma \ref{lem:dhk_est}. Now we use the same extension technique as in the proof of Lemma \ref{lem:hks_est}. By \eqref{eq:hk}, \eqref{eq:survs}, \eqref{eq:dhks} and Lemma \ref{lem:dhk_est}, for all $t \geq T_2/s$,
	\begin{align*}
	p_t^{\Gamma,s}(x,y) &= \frac{1}{\left(\psi^{-1}(1/s)\right)^d}p_{ts}^{\Gamma} \left( \frac{x}{\psi^{-1}(1/s)}, \frac{y}{\psi^{-1}(1/s)} \right) \\ &\approx \frac{1}{\left(\psi^{-1}(1/s)\right)^d} \P_{x/\psi^{-1}(1/s)} (\tau_{\Gamma}>st) \P_{y/\psi^{-1}(1/s)} (\tau_{\Gamma}>st) p_{st} \left( \frac{y-x}{\psi^{-1}(1/s)} \right) \\ &= \P_x(\tau_{\Gamma}^s>t)\P_y(\tau_{\Gamma}^s>t)p_t^s(x,y),
	\end{align*}
	with the implied constant independent of $s$.
\end{proof}
\begin{proposition}
	\label{prop:green}
	Assume {\bf A}. If $d\geq 2$ then for every $x \in \Rd$,
	\begin{equation}\label{eq:U_conv}
	\lim_{s \to \infty} U^s(x) = U^{\alpha}(x).
	\end{equation}
	Furthermore, for $d\geq 1$,
	\[
	\lim_{s \to \infty} G^s_{\Gamma}(x,\, y) =G^{\alpha}_{\Gamma}(x, y),
	\] 
	for  all $x$, $y \in \Gamma$.
\end{proposition}

\begin{proof}
	The first claim is an easy consequence of \cite[Corollary 3]{WCTGBT} for $d\geq 3$. Since $\alpha<2$, one can obtain the similar result also for $d=2$. The proof of the second claim is similar to that of Lemma~\ref{prop:dhk}. Recall that we adopt the notation $U^s(x,\, y) 
	= U^s(x-y) = \int_0^{\infty} p^s_t(x - y)\,d t$. Then by the Hunt formula we have 
	\begin{equation}
		\label{e:GreenG}
		G^s_{\Gamma}(x,\, y) = U^s(x,\, y) - \E_x^s \left[ U^s \left( X_{\tau_{\Gamma}},\, y \right) \right]. 
	\end{equation}
	
	We first refine \eqref{eq:U_conv} and prove that $U^s(u)$ in fact converges uniformly on compact subsets of $\R^d \setminus \{0\}$ to
	$U^{\alpha}(u)$ as $s \to \infty$. By \eqref{e:pts}, we have that
	\begin{equation}
		\label{e:smallt}
		\lim_{\epsilon \to 0} \int_0^\epsilon p_t^s(u) \,d t = 0 
	\end{equation}
	uniformly in $|u| \geq \delta$. Next, note that Lemma \ref{lem:hks_est} implies, for $T \geq T_1/s$,
	\[ \int_T^{\infty} p_t^s(u) \,d t \lesssim 
	\int_T^{\infty} \big(\psi^{-1,s}(1/t)\big)^{d} \,d t.\]
	Since $\psi^{-1} \in \calR_{1/\alpha}^0$, using \eqref{eq:psi_invs} and Potter's bounds for every $\eta>0$, we get the existence of $\delta>0$ with
	\begin{equation}\label{eq:tg4}\frac{\psi^{-1,s}(1/t)}{\psi^{-1,s}(1/u)}=\frac{\psi^{-1}(1/(st))}{\psi^{-1}(1/(su))}\leq 2\left(\frac{u}{t}\right)^{1/\alpha-\eta}, \quad t\geq u\geq 1/(s\delta). \end{equation}
		Hence,
	\[ \int_T^{\infty} p_t^s(u) \,d t \lesssim  T\big(\psi^{-1,s}(1/T)\big)^{d}\leq 2 T\big(\psi^{-1,s}(1)\big)^{d} T^{-d/\alpha+d\eta}\leq 4  T^{-d/\alpha+1+d\eta},\]
	which tends to $0$ as $T \to \infty$ uniformly in $u$, provided $d \geq 2$ and $\eta< 1/\alpha-1/d$.

	For $\epsilon \leq t \leq T$, we have from the proof of	Lemma~\ref{prop:dhk} that $\lim_{s \to \infty} p_t^s(u) = p^{\alpha}_t(u)$ uniformly in $\epsilon \leq t \leq T$ and $|u| \leq R$. It follows that	uniformly in $|u| \leq R$,
	\begin{equation}
		\label{e:middlet}
		\lim_{s \to \infty} \int_{\epsilon}^T p_t^s(u) \,dt = \int_{\epsilon}^T p_t^{\alpha}(u) \,d t.
	\end{equation}
 Combining \eqref{e:smallt} -- \eqref{e:middlet}, we get the  claimed uniform convergence.
	
	Now, 
	proceeding as in the proof of Lemma \ref{prop:dhk} we conclude that for every $R>2|y|$,
	\[
\lim_{s \to \infty} \E_x \big[ |X^s_{\tau^s_{\Gamma}}|\leq R; U^s ( X^s_{\tau^s_{\Gamma}},y ) \big] = \E_x \big[ |X^{\alpha}_{\tau^{\alpha}_{\Gamma}}|\leq R; U^{\alpha} ( X^{\alpha}_{\tau_{\Gamma}},y ) \big].
	\]
	The unimodality of the potential kernel and the fact that $\lim_{|u| \to \infty}U^{\alpha}(u)=0$ ends the proof in the case $d\geq 2$.
	
	If $d=1$, then $\Gamma=(0,\infty$); therefore, by Lemma \ref{lem:dhks_est} and \ref{lem:hks_est} together with \cite[Proposition 2.6]{KBTGMR-ptrf} and \eqref{eq:52},
	\[p_t^\Gamma(x,y)\lesssim \frac{1}{\sqrt{h^s(x)h^s(y)}}\frac{1}{t}\psi^{-1,s}(1/t), \quad x,y>0,\,\, t\geq T_2/s. \]
	Hence, for sufficiently large $s$, we have, by Proposition \ref{prop:nu_h_conv} and \eqref{eq:tg4},
	\[p_t^\Gamma(x,y)\lesssim (xy)^{\alpha/2}t^{-1/\alpha-1+\eta}, \quad x,y>0,\,\, t\geq1. \]
	This together with \eqref{e:pts} allow us to use the dominated convergence theorem and the claim in this case follow by Lemma \ref{prop:dhk}.

\end{proof}

\subsection{Uniform BHP and estimates of harmonic functions}\label{s.ubhp}
Let us turn to the boundary Harnack inequality and its consequences for harmonic functions. A very general version of BHP was proved by Bogdan, Kumagai and Kwaśnicki \cite{BKK15}. It was later simplified in \cite{TGMK} in the case of unimodal L\'{e}vy processes. Therefore, we will rather use the results from \cite{TGMK}, as they better serve our purpose. Next, using  ideas from \cite{KJ17} we will show that the boundary limits of ratios of harmonic functions exist and are uniform in the sense of \cite[Remark 3]{KJ17}. This observation will be crucial in the proof of Lemma \ref{lem:Mkernel}, which in turn is essential for Lemma \ref{lem:Gfs_limit} to hold.

First we prove the \emph{uniform} boundary Harnack inequality and generalize \cite[Remark 3]{KJ17} to the whole \emph{family} of L\'{e}vy processes $\big\lbrace {\mathbf X}^s \colon s\geq 1 \big\rbrace$. To this end we closely examine assumptions imposed in \cite{BKK15}, \cite{TGMK} and \cite{KJ17} and verify that constants appearing therein are in fact independent of $s$. First, let us show that the boundary Harnack inequality holds with the same constant for every ${\mathbf X}^s$. Recall that {\bf A1} holds true for all $\bfX^s$ for all $s \geq 1$ with the same parameters $M$ and $\beta$. Therefore, by Proposition \ref{prop:BHP}, $\cbhi$ is also independent on $s$. In particular, the \emph{uniform} boundary Harnack inequality holds: for all $s \geq 1$, $x_0 \in \partial D$ and two functions $f^s,g^s \geq 0$ which are regular harmonic with respect to $\bfX^s$ in $D \cap B(x_0,2r)$ and vanish on $D^c \cap B(x_0,2r)$,
	\begin{equation}\tag{BHP}\label{eq:BHP}
		\frac{f^s(x)}{g^s(x)} \leq \cbhi \frac{f^s(y)}{g^s(y)}, \quad x,y \in D \cap B(x_0,r),
	\end{equation}
	with $\cbhi =\cbhit^4$. For the existence of the boundary limits we verify the assumptions of \cite[Theorem 2]{KJ17}. Assumption (i) therein is satisfied for every unimodal L\'{e}vy process. Our goal is to prove that (ii) and (iii) hold uniformly in $s$. First, we check that for every $R>0$, $\cnu^s(r,R)$ satisfies
	\[
	\nu^s(t-r) \leq \cnu^s(r,R)\nu^s(t+r),
	\]
	for all $t>R$, may be chosen independent of $s$ (which justifies writing $\cnu^s=\cnu$). 
	Indeed, assume {\bf A1} and {\bf A2}. We  clearly have $\cnu^s(r,R) \geq 1$ for all $0<r<R$ and $s \geq 1$. Furthermore,
	\begin{equation*}
		\cnu^s(r,R) = \sup_{t \geq R} \frac{\nu^s(t-r)}{\nu^s(t+r)} \leq M \sup_{t \geq R} \bigg( \frac{t+r}{t-r} \bigg)^{d+\beta} =  \bigg( \frac{R+r}{R-r} \bigg)^{d+\beta}M,
	\end{equation*}	
	which proves the independence of the constant from $s$. In particular, by setting $R=2r$ we see that the condition (iii) holds uniformly in $s$. 
Finally, we note that condition (iv) actually holds independently of $s$ since the only ingredient in the proof of \cite[Theorem 1.11]{TGMK} is the scaling condition {\bf A1}, which is independent of $s$.
	
	Therefore, by \cite[Theorem 2 and Remark 1]{KJ17} we have the following.
        \begin{theorem} \label{thm:uni_lim} 
		Assume {\bf A}. Let $D$ be an open Lipschitz set, $x_0\in \partial D$ and $R>0$.  Then 
		$$\lim_{r\to 0^+}\sup_{s \geq 1} \sup_{f^s,g^s}\frac{\sup_{x\in D\cap B(x_0,r)}f^{s}(x)/g^{s}(x)}{\inf_{x\in D\cap B(x_0,r)}f^{s}(x)/g^{s}(x)}=1,$$
                where $\sup_{f^s,g^s}$ is taken over all non-negative functions $f^{s}, \, g^{s}$ that are  regular harmonic  in $D \cap B(x_0, R)$ with respect to  $\bfX^s$ and are equal to zero in $B(x_0, R) \setminus D$.
        \end{theorem}        
	\begin{proof}

			Without loss of generality we may and do assume that $x_0=0$. First of all, we show that  $x_0\in\partial D$ is an accessible boundary point (see \cite[Remark 1]{KJ17}) and then the proof will follow by inspection of the proof of Theorem 2 therein in the case of accessible boundary points in \cite[Section 4.3]{KJ17}. We remark that in this case the assumption (ii) of \cite[Theorem 2]{KJ17} is redundant.  Namely, since $D$ is a Lipschitz set, there exists $R>0$ and  an open right circular  cone $\tilde{\Gamma}$ with apex at $x_0$ such that $\tilde{\Gamma}_R\subset D$.  By  isotropy of $\bf{X}$  we may and do assume that the axis of $\tilde{\Gamma}$ is the line $\ell=\{t\one:t\in\mathbb{R}\}$. By $\kappa$-fatness of $\tilde{\Gamma}$ and due to the Pruitt bounds we have, for all $y\in\tilde{\Gamma_R}$ with $\mathrm{dist}(y, \ell)\leq \kappa/2 |y|$ that
\[ \E_y\tau^s_{D\cap B_R}\geq \E_y\tau^s_{\tilde{\Gamma_{R}}}\geq  \E_y\tau^s_{B(y,\kappa |y|/2)}\geq \frac{c(d)}{h^s(\kappa |y|/2)}\geq  \frac{c(d)\kappa^2}{4 h^s( |y|)}. \]
Hence, by Proposition \ref{prop:K} and the isotropy of $\bf{X}$, for large $s$ and small $\delta$,
\begin{align}\label{eq:tg7}\begin{aligned}
	\int_{D\cap B_R\setminus B_\delta}\E_y\tau^s_{D\cap B_R}\nu^s(y)dy&\geq c(d,M,\beta,\kappa) \int_{B_{R/2}\setminus B_{2\delta}}\frac{1}{h^s(|y|)}\frac{K^s(|y|)}{|y|^d}dy\\ &=c\int_{2\delta}^{R/2} \frac{1}{h^s(u)} \frac{K^s(u)}{u}\,du \\&=c(\ln h^s(2\delta)-\ln h^s(R/2)).\end{aligned}
\end{align}
Observe that, since $h^s$ blows up as $\delta \to 0^+$, $x_0$ is an accessible boundary point for $\bfX^s$. We will now prove that the blow-up is  uniform in $s$. Indeed, suppose that
\begin{equation}\label{eq:22}
\sup_{\delta < R} \inf_{s \geq 1}\frac{h^s(2\delta)}{h^s(R/2)}<\infty.
\end{equation}
Note that
\begin{equation}\label{eq:21}
\frac{h^s(2\delta)}{h^s(R/2)} = \frac{h\big(2\delta/\psi^{-1}(1/s)\big)}{h\big(R(2\psi^{-1}(1/s))\big)}.
\end{equation}
By Proposition \ref{prop:nu_h_conv} and continuity of $h$, for all small enough $\delta>0$ there is $s=s(\delta)$ with
\begin{equation*}
\limsup_{\delta \to 0^+} \frac{h^{s(\delta)}(2\delta)}{h^{s(\delta)}(R/2)}<\infty.
\end{equation*}
Moreover, we observe that the expression $\delta/\psi^{-1}(1/s(\delta))$ is bounded in $\delta$. Indeed, should the converse be true, \eqref{eq:21}, \eqref{eq:52} and \eqref{eq:64} would imply that
\[
\frac{h^{s(\delta)}(2\delta)}{h^{s(\delta)}(R/2)} \gtrsim \frac{\psi\big( \psi^{-1}(1/s(\delta))/\delta \big)}{\psi \big( \psi^{-1}(1/s(\delta))/R \big)} \gtrsim \delta^{-\alpha/2},
\]
where the last step follows by {\bf A2} together with Potter bounds on $\psi$, and we would arrive at a contradiction. Now, since $\psi^{-1}(1/s(\delta)) \leq 1$, there are two possible scenarios: either $s(\delta)$ is bounded or not. If it is bounded, then we observe that $h\big(R/(2\psi^{-1}(1/s))\big)$ in \eqref{eq:21} is bounded as well, but since $h\big( 2\delta/\psi^{-1}(1/s) \big)$ blows up, we get the contradiction with \eqref{eq:22}. If it is unbounded, then we infer that $h(2\delta/\psi^{-1}(1/s(\delta)))$ is bounded below by a positive constant, but $h\big(R/(2\psi^{-1}(1/s(\delta))\big)$ goes to $0$ as $\delta \to 0^+$ and again we get the contradiction. Hence, for every $R>0,$
\[
\sup_{\delta<R} \inf_{s \geq 1} \frac{h^s(2\delta)}{h^s(R/2)}=\infty,
\]
so $x_0$ is an accessible boundary point for $\bf{X}^s$ and the lower bounds above is independent of $s$. This implies that every constant that appears in Section 4.3 in \cite{KJ17} is in fact independent of $s$ and therefore the limit proven there is uniform with respect to $s$ and functions $f^{s}$ and $g^{s}$. 
	\end{proof}

\begin{lemma}\label{lem:upper_harm}
	Let $D$ be a $\kappa$-fat set. Assume {\bf A1}. There is a constant $C=C(d,M,\beta,\kappa)$ such that for all $s \geq 1$, $Q \in \partial D$, $r>0$ and non-negative functions $u^s$ regular harmonic in $D \cap B(Q,2r)$ with respect to $\bfX^s$ and vanishing on $D^c \cap B(Q,2r)$,
	\[
	u^s(x) \leq Cu^s(A_r(Q)), \quad x \in D \cap B(Q,r).
	\]
\end{lemma}
\begin{proof}
By Proposition \ref{prop:BHP} and discussion before \eqref{eq:BHP} we get
\[\frac{u^s(x)}{u^s(A_r(Q))}\leq \cbhit^2 \frac{\E_x\tau_{D\cap B(Q,4r/3)}}{\E_{A_r(Q)}\tau_{D\cap B(Q,4r/3)}}.\]
By Pruitt's estimates,
\[\E_x\tau_{D\cap B(Q,4r/3)}\leq \frac{c(d)}{h(4r/3)} \]
and
\[\E_{A_r(Q)}\tau_{D\cap B(Q,4r/3)}\geq \frac{c(d)}{h(\kappa r)}. \]
An application of \eqref{eq:h_doubling} yields the claim.
\end{proof}

\subsection{Uniform integrability}\label{s.ui}
Since a dilation of a Lipschitz set is Lipschitz, by
\eqref{e.p1} we get $\P_x(X^s_{\tau_{D}-} = X^s_{\tau_{D}}) = 0$. Thus, by \eqref{eq:IW}, and analogous results for $\bfX^\alpha$ we have
\begin{eqnarray}\label{eq:IWs1}
	1&=&\int_0^\infty \int_{\Gamma}\int_{\Gamma^c} p_u^{\Gamma,s}(x,y)\nu^s(y,z)\,dz\,dy\,du\\
	&\xrightarrow{s\to\infty}&1=\int_0^\infty \int_{\Gamma}\int_{\Gamma^c} p_u^{\Gamma,\alpha}(x,y)\nu^\alpha(y,z)\,dz\,dy\,du\quad \mbox{as $s\to \infty$.}\nonumber
\end{eqnarray}
Here and below $x\in \Gamma$ is fixed but arbitrary, and we consider the integrands as functions parametrized by $s\to \infty$.
Due to Lemma~\ref{prop:dhk} and \eqref{eq:nu_conv}, the integrands converge, too,
so by Vitali's convergence theorem, the integrand in \eqref{eq:IWs1} is uniformly integrable, see, e.g., \cite[Chapter~22]{MR3644418}. Conversely, by Vitali's theorem and uniform integrability, for each bounded function $f$,
\begin{eqnarray}\label{eq:IWsf}\begin{aligned}
	&&\int_0^\infty \int_{\Gamma}\int_{\Gamma^c} p_u^{\Gamma,s}(x,y)f(u,y,z)\nu^s(y,z)\,dz\,dy\,du\\
	&\xrightarrow{s\to\infty}&\int_0^\infty \int_{\Gamma}\int_{\Gamma^c} p_u^{\Gamma,\alpha}(x,y)f(u,y,z)\nu^\alpha(y,z)\,dz\,dy\,du.\end{aligned}
\end{eqnarray}
For instance, taking arbitrary $t\ge 0$ and letting $f={\bf 1}_{u>t}$, we get
\begin{equation}\label{e.sp}
	\P_x(\tau_\Gamma^s>t)\to \P_x(\tau_\Gamma^\alpha>t) \quad \mbox{as $s\to \infty$.}
\end{equation}
\begin{lemma}\label{l.zpk}
	For all $x \in \Gamma$ and $t>0$,
	\[
	\lim_{s \to \infty} P_t^{\Gamma,s} \kappa_{\Gamma}^s(x) = P_t^{\Gamma,\alpha} \kappa_{\Gamma}^{\alpha}(x).
	\] 
\end{lemma}
\begin{proof}
	By the semigroup property,
	\begin{eqnarray}
		 P_t^{\Gamma,s}\kappa_{\Gamma}^s(x)&=&
		\int_{\Gamma}\int_{\Gamma^c} p_t^{\Gamma,s}(x,y)\nu^s(y,z)\,dz\,dy\nonumber \\
		&=&
		\frac{2}{t}\int_0^{t/2} \int_{\Gamma} \int_{\Gamma}\int_{\Gamma^c} p_{t-u}^{\Gamma,s}(x,v)p_u^{\Gamma,s}(v,y)\nu^s(y,z)\,dz\,dy\,dv\,du.\label{e.ui2}
	\end{eqnarray}
	We claim that there is an integrable function $g\ge 0$ on $\Rd$ such that
	\begin{equation}\label{e.opgzg}
		p_{r}^{\Gamma,s}(x,v)\le g(v),\quad r\in (t/2,t), v\in \Rd,
	\end{equation}
	if $s$ is large enough. Indeed, by Lemma \ref{lem:hks_est} and \eqref{eq:nu_comp} there is $c=c(d,M,\beta,x)$ such that
	\[
	p_r^{\Gamma,s}(x,v) \leq p_r^{s}(x,v) \leq cp_r^s(v).  
	\]
	Next, again by Lemma \ref{lem:hks_est}, monotonicity of $\psi^{-1,s}$ and \eqref{eq:36},
	\[
	p_r^s(v) \leq c \left(\psi^{-1,s}(1/r)\right)^d \leq c \left(\psi^{-1,s}(2/t)\right)^d \leq 2c\left(\psi^{-1,\alpha}(2/t)\right)^d
	\]
	for $s$ large enough. Thus, with the aid of \eqref{e:pts} we conclude \eqref{e.opgzg}. 
	Note that
	\begin{eqnarray}
		&&\int_{\Gamma} \int_0^{\infty}\int_{\Gamma}\int_{\Gamma^c} g(v)p_u^{\Gamma,s}(v,y)\nu^s(y,z)\,dz\,dy\,du\,dv\label{e.sc}\\
		&&= \int_{\Gamma} g(v)\,dv\to \int_{\Gamma} g(v)\,dv\nonumber\\
		&&=\int_{\Gamma} \int_0^{\infty}\int_{\Gamma}\int_{\Gamma^c} g(v)p_u^{\Gamma,\alpha}(v,y)\nu^\alpha(y,z)\,dz\,dy\,du\,dv \nonumber.
	\end{eqnarray}
	Therefore the integrand in \eqref{e.sc} is uniformly integrable, and so is the integrand in \eqref{e.ui2}. The integrand in \eqref{e.ui2} converges everywhere, and the proof is completed by another application of Vitali's  theorem.
\end{proof}

\subsection{The proof of the main result}\label{sec:Martin}
We define the Martin kernel for isotropic $\alpha$-stable process at $0$ for $\Gamma$   by setting 
\[
M_0^{\alpha}(y) = \lim_{\Gamma \ni x \to 0} \frac{G_{\Gamma}^{\alpha}(x,y)}{G_{\Gamma}^{\alpha}(x,\one)}, \quad y \in \Rd \setminus \{0\}.
\]
A discussion about Martin kernel for stable case may be found in \cite{BoPaWa2018}. In particular, by \cite[eq. (2.18)]{BoPaWa2018} we get that $M_0^{\alpha}$ is homogeneous of degree  $\alpha-d - \tilde{\beta}$, i.e.
\begin{equation}\label{eq:K_hom}
	M_0^{\alpha}(x) = |x|^{\alpha-d-\tilde{\beta}}M_0^{\alpha}(x/|x|),
\end{equation}
where $0<\tilde{\beta}<\alpha$ is the homogeneity degree of the Martin kernel at infinity for $\Gamma$ (see \cite[(2.17)]{BoPaWa2018}). Furthermore, since the Martin kernel at infinity for $\Gamma$ is locally bounded in $\Rd$ (see Ba{\~n}uelos and Bogdan \cite[Theorem 3.2]{BB04}), by \cite[(2.18)]{BoPaWa2018} we get that $M_0^{\alpha}$ is locally bounded on $\Rd \setminus \{0\}$.
\begin{lemma}\label{lem:Mkernel}
	Assume {\bf A}. For every $y \in \Gamma$,
	\[
	\lim_{s \to \infty} \frac{G_{\Gamma}^s \big(x\psi^{-1}(1/s),y\big)}{G_{\Gamma}^s \big(x\psi^{-1}(1/s),\one \big)} = M_0^{\alpha}(y).
	\]
\end{lemma}
\begin{proof}
	Fix $y \in \Gamma$. We will prove the lemma by verifying that
	\begin{equation}\label{eq:3}
		\lim_{(\Gamma, \R_+)\ni (x,1/s) \to (0,0)} \frac{G_{\Gamma}^s \big(x,y\big)}{G_{\Gamma}^s \big(x,\one \big)}  = \lim_{\Gamma \ni x \to 0} \lim_{s \to \infty} \frac{G_{\Gamma}^s \big(x,y\big)}{G_{\Gamma}^s \big(x,\one \big)} = \lim_{\Gamma \ni x \to 0} \frac{G_{\Gamma}^{\alpha} \big(x,y\big)}{G_{\Gamma}^{\alpha} \big(x,\one \big)} = M_0^{\alpha}(y).
	\end{equation}
	To this end we will justify the application of the Moore-Osgood theorem. 
First we observe that in view of Proposition \ref{prop:green}, for every $x \in \Gamma$,
	\[
	\lim_{s \to \infty} \frac{G_{\Gamma}^s(x,y)}{G_{\Gamma}^s(x,\one)} = \frac{G_{\Gamma}^{\alpha}(x,y)}{G_{\Gamma}^{\alpha}(x,\one)}.
	\]
	Next, we note that Theorem \ref{thm:uni_lim} yields
	\[
	\lim_{r \to 0^+} \sup_{s\geq 1} \frac{\sup\limits_{x \in \Gamma_r} \dfrac{G_{\Gamma}^s (x,y)}{G_{\Gamma}^s (x,\one)}}{\inf\limits_{x \in \Gamma_r} \dfrac{G_{\Gamma}^s (x,y)}{G_{\Gamma}^s (x,\one)}} =1.
	\]
	That is, for every $\eta>0$ there exists $r=r(\eta)$ such that
	\begin{equation}\label{eq:18}
		\bigg|\sup_{x \in \Gamma_r} \frac{G_{\Gamma}^s(x,y)}{G_{\Gamma}^s(x,\one)} - \inf_{x \in \Gamma_r} \frac{G_{\Gamma}^s(x,y)}{G_{\Gamma}^s(x,\one)} \bigg| \leq \eta \inf_{x \in \Gamma_r} \frac{G_{\Gamma}^s(x,y)}{G_{\Gamma}^s(x,\one)}, \quad r<r(\eta), s\geq 1.
	\end{equation}
	We claim that $G_{\Gamma}^s(x,y)/G_{\Gamma}^s(x,\one)$ converges  as $\Gamma\ni x \to 0$, uniformly in $s \geq s_0$ for some $s_0\geq 1$. Fix $x_0 \in \Gamma_{|y|/2}$. By \eqref{eq:BHP},
	\[
	\sup_{x \in \Gamma_r} \frac{G_{\Gamma}^s(x,y)}{G_{\Gamma}^s(x,\one)} \leq c_1 \frac{G_{\Gamma}^s(x_0,y)}{G_{\Gamma}^s(x_0,\one)},
	\]
	if $r$ is small enough. Since both $y$ and $x_0=x_0(y)$ are fixed, Proposition \ref{prop:green} entails that
	\begin{equation}\label{eq:19}
		\sup_{x \in \Gamma_r} \frac{G_{\Gamma}^s(x,y)}{G_{\Gamma}^s(x,\one)} \leq 2c_1 \frac{G_{\Gamma}^{\alpha}(x_0,y)}{G_{\Gamma}^{\alpha}(x_0,\one)} = c_2,
	\end{equation}
	if $r>0$ is small enough and $s \geq s_0$, for $s_0$ large enough. We note that both $s_0$ and $c_2$ depend only on $y$ and in particular they do not depend on $r$. Now, if we fix $\epsilon>0$ and set $\eta = \epsilon/c_2$ then, using \eqref{eq:18} and \eqref{eq:19} for $|x_1|,|x_2|<r(\eta(\epsilon))=r(\epsilon)$, we obtain for $s\ge s_0$,
	\begin{align*}
		\bigg| \frac{G_{\Gamma}^s(x_1,y)}{G_{\Gamma}^s(x_1,\one)} - \frac{G_{\Gamma}^s(x_2,y)}{G_{\Gamma}^s(x_2,\one)}\bigg| &\leq \bigg|\sup_{x \in \Gamma_{r(\epsilon)}} \frac{G_{\Gamma}^s(x,y)}{G_{\Gamma}^s(x,\one)} - \inf_{x \in \Gamma_{r(\epsilon)}} \frac{G_{\Gamma}^s(x,y)}{G_{\Gamma}^s(x,\one)} \bigg| \\ &\leq \frac{\epsilon}{c_2} \inf_{x \in \Gamma_{r(\epsilon)}} \frac{G_{\Gamma}^s(x,y)}{G_{\Gamma}^s(x,\one)} \\  &\leq \frac{\epsilon}{c_2} \sup_{x \in \Gamma_{r(\epsilon)}} \frac{G_{\Gamma}^s(x,y)}{G_{\Gamma}^s(x,\one)}  \\ &\leq \frac{\epsilon}{c_2} \cdot c_2=\epsilon,
	\end{align*}
	and the claim is proved. Thus, by the Moore-Osgood theorem \cite[Chapter VII]{Graves} we obtain \eqref{eq:3} and the lemma follows immediately.
\end{proof}

\begin{lemma}\label{lem:Gfs_limit}
	Assume {\bf A}. Let $f^s$ be a family of measurable non-negative functions which are uniformly  bounded on $\Gamma_r$ for each $r \geq 1$ and $s \geq s_0$ with some $s_0 \geq 1$. Suppose that $f^s \to f^{\alpha}$ a.e. $\Gamma$ as $s \to \infty$. We also assume that $G^sf^s$ are uniformly bounded on $\Gamma$ and there is $x_0 \in \Gamma$ such that $\lim_{s \to \infty} G_{\Gamma}^sf^s(x_0) = G_{\Gamma}^{\alpha} f^{\alpha}(x_0)$. Then $\int_{\Gamma} M_0^{\alpha}(y) f^{\alpha}(y)\,dy<\infty$ and for every $x\in\Gamma$,
	\begin{equation}\label{eq:Gfs_limit}
		\lim_{s\to \infty} \frac{G^{s}_{\Gamma}f^s(\psi^{-1}(1/s)x)}{G^{s}_{\Gamma}(\psi^{-1}(1/s)x,\one)} = \int_{\Gamma} M_0^{\alpha}(y)f^{\alpha}(y)\:dy.
	\end{equation}
\end{lemma}
\begin{proof}
	We closely follow the proof of \cite[Lemma 3.5]{BoPaWa2018} and adapt it to our setting. Fix $x \in \Gamma$, let $0<\delta<R$ and denote $x_s = x\psi^{-1}(1/s)$. By \eqref{eq:BHP} we have, for large enough $s>s_0=s_0(\delta)$ and some $x_1 \in \Gamma_{\delta/2}$,
	\begin{equation}\label{eq:5}
		\frac{G^s (x_s,y)}{G^s (x_s,\one)} \leq \cbhi \frac{G^s(x_1,y)}{G^s(x_1,\one)}, \quad |y|>\delta.
	\end{equation}
	Next, by Proposition \ref{prop:green} we have $G_{\Gamma}^s(x_1,\one) \geq \tfrac12 G_{\Gamma}^{\alpha}(x_1,\one)$ for sufficiently large $s \geq s_0$. Furthermore, by \eqref{eq:K_hom} and local boundedness on $\Rd \setminus \{0\}$ we get that $M_0^{\alpha}$ is integrable at the origin; thus, by \eqref{eq:5} with $\delta=1$, the assumption of uniform boundedness of $f^s$ on $\Gamma_1$ and Fatou's lemma,
	\[
	\int_{\Gamma} M_0^{\alpha}(y) f^{\alpha}(y)\,dy \leq c \bigg(\int_{\Gamma_1}M_0^{\alpha}(y)\,dy + \frac{1}{G_{\Gamma}^{\alpha}(x_1,y)} \liminf_{s \to \infty} \int_{\Gamma} G_{\Gamma}^s (x_1,y) f^s(y)\,dy  \bigg) < \infty,
	\] 
	which proves that the right-hand side of \eqref{eq:Gfs_limit} is finite.
	
	Now, we split the integral as follows:
	\begin{align*}
		\frac{G_{\Gamma}^s f^s(x_s)}{G_{\Gamma}^s(x_s,\one)} &= \int_{\Gamma} \frac{G_{\Gamma}^s(x_s,y)}{G_{\Gamma}^s(x_s,\one)} f^s(y)\,dy = \bigg(\int_{\Gamma_{\delta}} + \int_{\Gamma_R \setminus \Gamma_{\delta}} + \int_{\Gamma \setminus \Gamma_R}\bigg) \frac{G_{\Gamma}^s(x_s,y)}{G_{\Gamma}^s(x_s,\one)} f^s(y)\,dy \\ &=:I_1(s)+I_2(s)+I_3(s).
	\end{align*}
	For $d\geq 2$, by \eqref{eq:5} and Proposition \ref{prop:green} we have, for $s$ large enough and $|y|> \delta$,
	\[
	\frac{G_{\Gamma}^s(x_s,y)}{G_{\Gamma}^s(x_s,\one)} 
	\leq c \frac{U^s(x_1-y)}{G_{\Gamma}^{\alpha}(x_1,\one)} \leq c \frac{U^s(\delta/2)}{G_{\Gamma}^{\alpha}(x_1,\one)} \leq c \frac{U^{\alpha}(\delta/2)}{G_{\Gamma}^{\alpha}(x_1,\one)}
	\] 
	for some $c = c(\delta)$. A similar bound, for $d=1$, is a consequence of Lemma \ref{lem:Ub}, Proposition \ref{prop:nu_h_conv} and the fact that $u \mapsto u^2h(u)$ is non-decreasing.   Therefore, since $f_s$ are uniformly bounded on $\Gamma_R$, by the dominated convergence theorem and Lemma \ref{lem:Mkernel},
	\begin{equation}\label{eq:16}
		\lim_{s \to \infty} I_2(s) = \lim_{s \to \infty} \int_{\Gamma_R \setminus \Gamma_{\delta}} \frac{G_{\Gamma}^s(x_s,y)}{G_{\Gamma}^s(x_s,\one)} f^s(y)\,dy = \int_{\Gamma_R \setminus \Gamma_{\delta}} M_0^\alpha(y)
		f^{\alpha}(y)\,dy.
	\end{equation}
	Next, let $x_0$ be such that
	\[
	\lim_{s \to \infty} G_{\Gamma}^s f^s(x_0) = G_{\Gamma}^{\alpha} f^{\alpha}(x_0).
	\]
	We may and do assume that $R>2|x_0|$. Again by \eqref{eq:BHP}, for $s$ large enough,
	\[
	\frac{G_{\Gamma}^s(x_s,y)}{G_{\Gamma}^s(x_s,\one)} \leq \cbhi \frac{G_{\Gamma}^s(x_0,y)}{G_{\Gamma}^s(x_0,\one)}, \quad |y| \geq R.
	\]
	Therefore, by Proposition \ref{prop:green},
	\begin{align*}
		I_3(s) &=\int_{\Gamma \setminus \Gamma_R} \frac{G_{\Gamma}^s(x_s,y)}{G_{\Gamma}^s(x_s,\one)} f^s(y)\,dy \leq \cbhi \int_{\Gamma \setminus \Gamma_R} \frac{G_{\Gamma}^s(x_0,y)}{G_{\Gamma}^s(x_0,\one)} f^s(y)\,dy \\ &\leq \frac{c}{G_{\Gamma}^{\alpha}(x_0,\one)} \int_{\Gamma \setminus \Gamma_R} G_{\Gamma}^s(x_0,y)f^s(y)\,dy,
	\end{align*}
	for large enough $s\geq s_0$. Denote $c_1=c/G^{\alpha}(x_0,\one)$. By the Fatou lemma,
	\begin{align}
		\limsup_{s \to \infty} I_3(s) &\leq c_1 \bigg( \limsup_{s \to \infty} G_{\Gamma}^s f^s(x_0) - \liminf_{s \to \infty} \int_{\Gamma_R} G_{\Gamma}^s(x_0,y)f^s(y)\,dy \bigg) \nonumber \\ &\leq c_1 \bigg( G_{\Gamma}^{\alpha} f^{\alpha}(x_0) - \int_{\Gamma_R} G_{\Gamma}^{\alpha}(x_0,y)f^{\alpha}(y)\,dy \bigg) \label{eq:6} \\ &= c_1 \int_{\Gamma \setminus \Gamma_R} G_{\Gamma}^{\alpha}(x_0,y)f^{\alpha}(y)\,dy. \nonumber
	\end{align}
	
	It remains to estimate $I_1(s)$. Note that, by symmetry, $y \mapsto G_{\Gamma}^s(v,y)$  is regular harmonic on $\Gamma_{\delta}$ when $v \in \Gamma \setminus \Gamma_{2\delta}$.  Using \eqref{eq:BHP} we obtain, for $y \in \Gamma_{\delta}$ and $v \in \Gamma \setminus \Gamma_{2\delta}$,
	\begin{equation}\label{eq:7}
		\frac{G_{\Gamma}^s(v,y)}{G_{\Gamma}^s(y,2\delta \cdot \one)} \leq \cbhi \frac{G_{\Gamma}^s(v,\delta/2 \cdot \one)}{G_{\Gamma}^s(\delta/2 \cdot \one,2\delta \cdot \one)}.
	\end{equation}
	Therefore, for $s$ large enough and $y \in \Gamma_{\delta}$, by \eqref{eq:7} and \eqref{eq:G_harm},
	\begin{align}\label{eq:11}\begin{aligned}
			G_{\Gamma}^s(x_s,y) &= G_{\Gamma_{2\delta}}^s(x_s,y) + \E_{x_s} G_{\Gamma}^s\big( X_{\tau^s_{\Gamma_{2\delta}}}^s,y \big) \\ &\leq G_{\Gamma_{2\delta}}^s(x_s,y) + c \E_{x_s} \Big[ G_{\Gamma}^s \big( X^s_{\tau^s_{\Gamma_{2\delta}}},\delta/2 \cdot \one \big) \Big] \cdot \frac{G_{\Gamma}^s(y,2\delta \cdot \one)}{G_{\Gamma}^s (\delta/2 \cdot \one,2\delta \cdot \one)} \\ &\leq G_{\Gamma_{2\delta}}^s(x_s,y) + c G_{\Gamma}^s(x_s,\delta/2 \cdot \one) \cdot \frac{G_{\Gamma}^s(y,2\delta \cdot \one)}{G_{\Gamma}^s (\delta/2 \cdot \one,2\delta \cdot \one)}.
	\end{aligned}\end{align}
	Thus, by uniform boundedness of $f_s$ on $\Gamma_{\delta}$ and \eqref{eq:11},
	\begin{equation}\label{eq:12}
		\int_{\Gamma_{\delta}} G_{\Gamma}^s(x_s,y)f^s(y)\,dy \leq c \bigg( \int_{\Gamma_{\delta}} G_{\Gamma_{2\delta}}^s(x_s,y)\,dy + \frac{G_{\Gamma}^s(x_s,\delta/2 \cdot \one)}{G_{\Gamma}^s (\delta/2 \cdot \one,2\delta \cdot \one)} \int_{\Gamma_{\delta}} G_{\Gamma}^s(y,2\delta \cdot \one)\,dy \bigg).
	\end{equation}
	Lemma \ref{lem:upper_harm} implies
	\begin{equation}\label{eq:15}
		\int_{\Gamma_{\delta}} G_{\Gamma}^s (y,2\delta \cdot \one) \,dy \leq c_5\delta^d G_{\Gamma}^s(\delta/2 \cdot \one,2\delta \cdot \one).
	\end{equation}
	Indeed, observe that there is $\tilde{\kappa} \leq \kappa$ such that $B(1/2 \cdot \one,\tilde{\kappa}) \subset \Gamma_1$, thus after a possible change of $\kappa$ we may and do assume that $A_1(0)=1/2 \cdot \one$. Since $\Gamma$ is a cone, it follows immediately that $A_\delta(0) = \delta/2 \cdot \one$.  Moreover, observe that $\Gamma_{\delta}$ is created from $\Gamma_1$ by scaling and the Lipschitz constant as well as $\kappa$ are not affected by the operation; therefore, $c_5$ is independent of $\delta$.
	Furthermore, we have by Proposition \ref{prop:BHP},
	\begin{align}\label{eq:13}
		\frac{\int_{\Gamma_{\delta}} G_{\Gamma_{2\delta}}^s(x_s,y)\,dy}{ G_{\Gamma}^s(x_s,\one)} &\leq \cbhit\frac{\E_{x_s}\tau^s_{\Gamma_{2\delta}}}{\E_{x_s}\tau^s_{\Gamma_{2\delta}}\int_{\Gamma^c_{5\delta/2}}G_{\Gamma}^s(y,\one)\nu^s(y)dy}
	\end{align}
	Since, by Propositions \ref{prop:BHP}, \ref{prop:nu_h_conv}  and \ref{prop:green}, for $s$ large and $y\in \Gamma_{1/2}$ (see \eqref{eq:BHP}),
		\[G_{\Gamma}^s(y,\one) \approx \E_y\tau^s_{\Gamma_{2/3}} \int_{\Gamma^c_{5/6}}G^s_{\Gamma}(w,\one)\nu^s(w)\,dw\approx  \E_y\tau^s_{\Gamma_{2/3}},  \]
		we obtain
		\begin{equation}\label{eq:tg17}\int_{\Gamma^c_{5\delta/2}}G_{\Gamma}^s(y,\one)\nu^s(y)dy\geq c \int_{\Gamma_{1/2}\setminus \Gamma_{3\delta}}\E_y\tau^s_{\Gamma_{2/3}}\nu^s(y)dy\geq c \ln \frac{h^s(3\delta)}{h^s (1/2)} ,
		\end{equation}
		where the last inequality is a consequence of \eqref{eq:tg7}.
	Thus, using  \eqref{eq:12}, \eqref{eq:15} and \eqref{eq:13} together with \eqref{eq:tg17} we infer that
	\[
	I_1(s) \leq c_6 \bigg( 
	\frac{1}{ \ln \frac{h^s(3\delta)}{h^s (1/2)}}+ \delta^d \frac{G_{\Gamma}^s(x_s,\delta/2 \cdot \one)}{G_{\Gamma}^s(x_s,\one)} \bigg).
	\]
	Lemma \ref{lem:Mkernel} and the fact that $h^{\alpha}(\rho)=c\rho^{-\alpha}$ 
	now yield
	\[
	\limsup_{s \to \infty} I_1(s) \leq c_7 \bigg(\frac{1}{-\ln\delta}+ 
	\delta^d M_0^{\alpha}(\delta/2 \cdot \one) \bigg).
	\]
	Finally, using the homogeneity of $M_0^{\alpha}$ \eqref{eq:K_hom} 
	we get
	\begin{align}\label{eq:17}\begin{aligned}
			\limsup_{s \to \infty} I_1(s) &\leq c_7 \bigg( \frac{1}{-\ln\delta}+ 
			\delta^{\alpha-\tilde{\beta}} M_0^{\alpha}(\one)\bigg). 
	\end{aligned}\end{align}
	Now, by \eqref{eq:16}, \eqref{eq:6}, \eqref{eq:17} and the Fatou lemma,
	\begin{align*}
		&\int_{\Gamma_R \setminus \Gamma_{\delta}} M_0^{\alpha}(y)f^{\alpha}(y)\,dy \leq \liminf_{s \to \infty} \frac{G_{\Gamma}^s f^s(x_s)}{G_{\Gamma}^s(x_s,\one)} \leq \limsup_{s \to \infty} \frac{G_{\Gamma}^s f^s(x_s)}{G_{\Gamma}^s(x_s,\one)} \\ &\leq c_8 \bigg(\frac{1}{-\ln\delta}+ \delta^{\alpha-\tilde{\beta}}\bigg) + \int_{\Gamma_R \setminus \Gamma_{\delta}} M_0^{\alpha}(y)f^{\alpha}(y)\,dy + c_1 \int_{\Gamma \setminus \Gamma_R} G_{\Gamma}^{\alpha}(x_0,y)f^{\alpha}(y)\,dy.
	\end{align*}
	Recalling that $\alpha > \tilde{\beta}$ and letting $\delta\to 0$ and $R\to \infty$ we end the proof.
\end{proof}
\begin{theorem}\label{thm:1}
	Assume {\bf A}. Let $x_s = x\psi^{-1}(1/s)$. For every $t>0$ we have
	\[
	\lim_{s \to \infty} \frac{\P_{x_s}(\tau_{\Gamma}^s>t)}{G_{\Gamma}^s(x_s,\one)} = C_t,
	\]
	where $C_t\in (0,\infty)$ is given by
	\[
	C_t = \int_{\Gamma} \int_{\Gamma} M_0^{\alpha}(y)p_t^{\Gamma,\alpha}(y,z)\kappa_{\Gamma}^{\alpha}(z)\,dz\,dy = \int_{\Gamma} M_0^{\alpha}(y)P_t^{\Gamma,\alpha} \kappa_{\Gamma}^{\alpha}(y)\,dy.
	\]
\end{theorem}
\begin{proof}
	Fix $t>0$. We verify that $f^s(x) = P_t^{\Gamma,s}\kappa_{\Gamma}^s(x)$ satisfies the assumptions of Lemma \ref{lem:Gfs_limit}; then the theorem will follow immediately. Indeed, first observe that pointwise convergence is a claim of Lemma \ref{l.zpk}. Fix $r>0$ and let us verify the uniform boundedness of $f^s$ on $\Gamma_r$. By Lemma \ref{lem:dhks_est} and {\bf A1} we have
	\begin{align*}
		P^{\Gamma,s}_1\kappa^{s}(x)&\approx \P_x(\tau^s_{\Gamma}>t)\int_{\Gamma}\P_y(\tau^s_{\Gamma}>t)p^s_t(y)\kappa^{s}(y)dy
	\end{align*}
	for $s$ large enough, with the comparability  constant dependent only on $d,M,\beta$ and $r$. By \eqref{eq:surv_id}, we have $G^{s}_\Gamma P^{\Gamma,s}_t\kappa^{s}(x)\leq 1$ for all $x \in \Rd$. Since the survival probability is bounded from above by $1$, it remains to find the upper bound for the integral.  If $\delta = \delta_{\Gamma}(\one)$, then by Pruitt's estimates, Proposition \ref{prop:nu_h_conv} and the argument from \eqref{eq:t1},
	\begin{align*}
		1&\geq c\int_{B(\textbf{1},r)}G^{s}_\Gamma(\one,z)\int_{\Gamma}\P_y(\tau^s_{\Gamma}>t)p^s_t(y)\kappa^{s}(y)\,dy\,dz\\&\geq c\E \tau_{B_{\delta}} \int_{\Gamma}\P_y(\tau^s_{\Gamma}>t)p^s_t(y)\kappa^{s}(y)dy  \\ & \geq \frac{c}{h^s(\delta)} \int_{\Gamma}\P_y(\tau^s_{\Gamma}>t)p_t^s(y)\kappa^{s}(y)dy \\ &\geq \frac{c}{h^{\alpha}(\delta)} \int_{\Gamma}\P_y(\tau^s_{\Gamma}>t)p_t^s(y)\kappa^{s}(y)dy.
	\end{align*}
Therefore the functions $P^{\Gamma,s}_t\kappa^{s}$ are uniformly bounded on $\Gamma_r$ for every $r>0$.
	
	Finally, by Corollary \ref{cor:surv_lim} we have $\lim_{s \to \infty} G_{\Gamma}^s f^s(x) = G_{\Gamma}^{\alpha} f^{\alpha}(x)$ for every $x \in \Gamma$, hence an application of Lemma \ref{lem:Gfs_limit} ends the proof.
\end{proof}
The following theorem refines \cite[Theorem 3.3]{BoPaWa2018}, which is a special case for $\bfX$ being the isotropic $\alpha$-stable L\'{e}vy process in $\Rd$, but we note that part of the proof relies on the consequences of \cite[Theorem 1.1]{BoPaWa2018}.

	\begin{theorem}\label{thm:Yaglom_density}Let $X$ be a pure-jump isotropic unimodal L\'{e}vy process. Assume {\bf A}. Let $\Gamma$ be a Lipschitz cone. Then  the following limit exists
		\[
		\lim_{s \to \infty} \frac{p_1^{\Gamma,s}(x_s,y)}{\P_{x_s}(\tau_{\Gamma}^s>1)} = n^{\alpha}(y),
		\]
		where $n^{\alpha}(y)$ is the function from \eqref{eq:limit_n}.
	\end{theorem}

\begin{proof}
	The proof follows directly the proof of \cite[Theorem 3.3]{BoPaWa2018}. Fix $x \in \Gamma$ and $t>0$. Consider the  family of measures defined as follows,
	\begin{equation}\label{eq:41}
		\mu_t^s(A) = \int_{A} f_t^s(y)\,dy = \frac{\int_A p_t^{\Gamma,s}(x_s,y)\,dy}{\P_{x_s} (\tau_{\Gamma}^s>t)}, \quad A \subset \Rd,
	\end{equation}
	with $x_s = \psi^{-1}(1/s)x$. We claim that the family $\{\mu_t^s \colon s \geq s_0 \}$, where $s_0$ will be specified later in the proof, is tight. Indeed, proceeding as in the proof of Proposition \ref{prop:nu_h_conv} and applying \eqref{eq:36} we have that
	\begin{equation}\label{eq:48}
		\left( \psi^{-1,s}(1/t)\right)^d \wedge t\nu^s(y) \lesssim 1 \wedge \nu^s(y) \lesssim 1 \wedge |y|^{-d-\alpha/2}
	\end{equation}
	$s \geq s_0$, with the implied constant independent of $s$. It follows from Lemma \ref{lem:hks_est} and \ref{lem:dhks_est}, and \eqref{eq:nu_comp} that we may bound the densities $f_t^s$ by a fixed integrable function, i.e., 
	\begin{equation}\label{eq:35}
		\frac{p_t^{\Gamma,s}(x_s,y)}{\P_{x_s}(\tau_{\Gamma}^s>t)} \approx \P_y(\tau_{\Gamma}^s>t) p_t^s(y-x_s) \lesssim p_t^s(y) \lesssim 1 \wedge |y|^{-d-\alpha/2}, \quad s\geq s_0, \ y \in \Rd.
	\end{equation}
	Recall that $t$ is fixed and the implied constant depends only on $d$, $M$ and $\beta$. Consider an arbitrary sequence $\{s_n\}$ with $\lim_{n \to \infty}s_n=\infty$. By the Prokhorov theorem, there is a subsequence $\{s_{n_k}\}$ such that $\mu_t^{s_{n_k}}$ converges weakly to a probability measure $\mu_t$ as $k \to \infty$.
	
	Let $\phi \in C_c^{\infty}(\Gamma)$ and set $u_{\phi}^s = -\calL^s \phi$. For every $s \geq 1$, $u_{\phi}^s$ is bounded, continuous, and, in view of Proposition \ref{prop:G_anihilation}, $G_{\Gamma}^su_{\phi}^s(x)=\phi(x)$ for $x \in \Rd$. By \eqref{eq:47},
	\begin{equation}\label{eq:33}
		\big| \calL^s\phi(y) \big| \lesssim 1 \wedge \nu^s(y), \quad y \in \Rd.
	\end{equation}

	In view of Lemma \ref{lem:hks_est} and \eqref{eq:36}, $|u_{\phi}^s(x)| \leq c p_1^s(x)$. Note that $c$ here may depend on $s$ but it is irrelevant for the proof of \eqref{eq:78}. It follows that
	\[
	P_t^{\Gamma,s}|u_{\phi}^s|(x)\leq c p_{t+1}^s(x),
	\]
	and consequently, for every $x \in \Gamma$,
	\[
	G_{\Gamma}^s P_t^{\Gamma,s} |u_{\phi}^s|(x) \leq c\int_{\Rd} G_{\Gamma}^s(x,y)p_{t+1}^s(y)\,dy < \infty,
	\]
	where the last inequality follows from Lemma \ref{lem:hks_est} and \eqref{eq:weightedL1}. By the Fubini-Tonelli theorem, for every $s\geq 1$,
	\begin{equation}\label{eq:78}
		G_{\Gamma}^s P_t^{\Gamma,s} u_{\phi}^s(x) = P_t^{\Gamma,s} G_{\Gamma}^s u_{\phi}^s(x) = P_t^{\Gamma,s} \phi(x).
	\end{equation}
	
	Next, observe that by {\bf A1}, \eqref{eq:nu_conv} and \eqref{eq:48}, for $s$ large enough,
	\[
	\big| \phi(x+y)-\phi(x)\big|\nu^s(y) \leq \norm{\phi}_{C^2(\Rd)}(|y|^2 \wedge 1)\nu^s(y) \lesssim |y|^{2-d-\beta}\ind_{B_1}(y) +  |y|^{-d-\alpha/2}\ind_{B_1^c}(y),
	\]
	with the implied constant independent of $s$. Thus, the dominated convergence theorem entails that for every $y \in \Gamma$,
	\[
	\lim_{s \to \infty} u_{\phi}^s(y) = u_{\phi}^{\alpha}(y) = -\Delta^{\alpha/2} \phi(y).
	\]
	Moreover, if $R>0$ is such that $\dist(B_R^c,\supp \phi) \geq 1$ and $x \in B_R^c$ then in view of \eqref{eq:48}  we may refine \eqref{eq:33} so that
	\[
	\big| \calL^s \phi(y) \big| \leq c\left( 1 \wedge |y|^{-d-\alpha/2}\right), \quad y \in \Rd,	
	\]
	for $s$ large enough with $c=c(d,M,\beta)$.	Thus, by Lemma \ref{lem:hks_est}, \eqref{eq:36} and the dominated convergence theorem, for every $x \in \Gamma$,
	\[
	\lim_{s \to \infty} P_t^{\Gamma,s}u_{\phi}^s(x) = P_t^{\Gamma,\alpha}u_{\phi}^{\alpha}(x).
	\]
	The same argument yields that $P_t^{\Gamma,s}u_{\phi}^s$ are uniformly bounded and that $G_{\Gamma}^s P_t^{\Gamma,s}u_{\phi}^s= P_t^{\Gamma,s}\phi$ are uniformly bounded. We also get that
	\[
	\lim_{s \to \infty} G_{\Gamma}^s P_t^{\Gamma,s}u_{\phi}^s(x) = \lim_{s \to \infty} P_t^{\Gamma,s}\phi(x) = P_t^{\Gamma,\alpha}\phi(x) = G_{\Gamma}^{\alpha} P_t^{\Gamma,\alpha}u_{\phi}^{\alpha}(x), \quad x \in \Gamma.
	\]
	Thus, by Lemma \ref{lem:Gfs_limit},
	\[
	\lim_{s \to \infty} \frac{P_t^{\Gamma,s}\phi(x_s)}{G_{\Gamma}^s(x_s,\one)} = \lim_{s \to \infty} \frac{G_{\Gamma}^s P_t^{\Gamma,s} u_{\phi}^s(x_s)}{G_{\Gamma}^s(x_s,\one)} = \int_{\Gamma} M_0^{\alpha}(y) P_t^{\Gamma,\alpha} u_{\phi}^{\alpha}(y)\,dy,
	\]
	If we denote $\mu_t^s(\phi) = \int_{\Gamma}\phi(y)\,\mu_t^s(dy)$, then by Theorem \ref{thm:1} and the identity above we conclude that there is a finite limit
	\[
	\lim_{s \to \infty} \mu_t^s(\phi) = \lim_{s \to \infty} \frac{P_t^{\Gamma,s}\phi(x_s)}{\P_{x_s}(\tau_{\Gamma}^s>t)} = \frac{\int_{\Gamma} M_0^{\alpha}(y)P_t^{\Gamma,\alpha}u_{\phi}^{\alpha}(y)\,dy}{\int_{\Gamma}M_0^{\alpha}(y)P_t^{\Gamma,\alpha}\kappa_{\Gamma}^{\alpha}(y)\,dy}.
	\] 
	In particular, $\mu_t(\phi) = \lim_{k \to \infty}\mu_t^{s_{n_k}}(\phi)$ does not depend on the choice of subsequence $s_{n_k}$. Therefore, $\mu_t^s$ converges weakly to $\mu_t$ as $s \to \infty$. 
	
	Moreover, we observe that for $t=1$ the limit measure $\mu_1$ is exactly the same as the one in the proof of \cite[Theorem 3.3]{BoPaWa2018}. By a repetition of the arguments in the proofs of \cite[(3.16), Theorem 3.1 and Theorem 3.3]{BoPaWa2018} one can conclude that the same holds true for every $t>0$, i.e., the measures
	\[
	\widetilde{\mu}_{x,t}(A) = \frac{\int_A p_t^{\Gamma,\alpha}(x,y)\,dy}{\P_x(\tau_{\Gamma}^{\alpha}>t)}	
	\]
	in \cite{BoPaWa2018} converge weakly as $\Gamma \ni x \to 0$ to the same measure $\mu_t$.
	It is therefore appropriate to use the notation $\mu_t^{\alpha}$ instead of $\mu_t$. Moreover, also from the proof of \cite[Theorem 3.3]{BoPaWa2018} we may conclude that the limit
	\[
	f_t^{\alpha}(y) = \lim_{\Gamma \ni x \to 0} \frac{p_t^{\Gamma,\alpha}(x,y)}{\P_x(\tau_{\Gamma}^{\alpha}>t)}
	\] 
	exists and is the density function of the measure $\mu_t^{\alpha}$. Note here that $f_1^{\alpha}=n^{\alpha}$. 	
	We now prove that
	\begin{equation}\label{eq:46}
		\lim_{s \to \infty} f_{1}^s(y) = \lim_{s \to \infty} \frac{p_1^{\Gamma,s}(x_s,y)}{\P_{x_s}(\tau_{\Gamma}^s>1)} = n^{\alpha}(y), \quad y \in \Gamma.
	\end{equation}
	Indeed, fix $y \in \Gamma$ and we denote $\phi_y^s(\,\cdot\,) = p_{1/2}^{\Gamma,s}(\,\cdot\,,y)$. By the Chapman-Kolmogorov equation,
	\[
	p_1^{\Gamma,s}(u,y) = \int_{\Gamma} p_{1/2}^{\Gamma,s}(u,z) p_{1/2}^{\Gamma,s}(z,y)\,dz = P_{1/2}^{\Gamma,s}\phi_y^s(u), \quad u \in \Rd.
	\]
	Thus,
	\[
	\frac{p_1^{\Gamma,s}(x_s,y)}{\P_{x_s}(\tau_{\Gamma}^s>1)} = \frac{P_{1/2}^{\Gamma,s} \phi_y^s(x_s)}{\P_{x_s}(\tau_{\Gamma}^s>1)} = \mu_{1/2}^s (\phi_y^s).
	\]
	We claim that
	\begin{equation}\label{eq:42}
		\lim_{s \to \infty} \mu_{1/2}^s(\phi_y^s) = \mu_{1/2}^{\alpha}(\phi_y^{\alpha}), 
	\end{equation}
	where $\phi_y^{\alpha}(\,\cdot\,) = p_{1/2}^{\Gamma,\alpha}(\,\cdot\,,y)$. Since, by the weak convergence of measures, $\lim_{s \to \infty} \mu_{1/2}^s(\phi_y^{\alpha}) = \mu_{1/2}^{\alpha}(\phi_y^{\alpha})$, it remains to prove that $\mu_{1/2}^s(\phi_y^s - \phi_y^{\alpha}) \to 0$ as $s \to \infty$. Since, in view of Lemma \ref{lem:hks_est} and \eqref{eq:36}, $\phi_y^s$ and $\phi_y^{\alpha}$ are uniformly bounded by a constant $c_1$ independent of $s$, the dominated convergence theorem implies that
	\[
	\lim_{s \to \infty} \mu_{1/2}^{\alpha} \big( \phi_y^s - \phi_y^{\alpha} \big) = 0.
	\]
	Thus, in order to get \eqref{eq:42} we need to show that
	\begin{equation}\label{eq:45}
		\lim_{s \to \infty}  \big( \mu_{1/2}^s \big( \phi_y^s -\phi_y^{\alpha}\big) - \mu_{1/2}^{\alpha} \big( \phi_y^s -\phi_y^{\alpha}\big) \big) = 0.
	\end{equation}
	Recall that $\mu_{1/2}^s(dz) = f_{1/2}^s(z)\,dz$ and $\mu_{1/2}^{\alpha}(dz) = f_{1/2}^{\alpha}(z)\,dz$. Thus, we may write
	\[
	\Big| \mu_{1/2}^s \big( \phi_y^s -\phi_y^{\alpha}\big) - \mu_{1/2}^{\alpha} \big( \phi_y^s -\phi_y^{\alpha}\big) \big| \leq \int_{\Gamma} \Big| \phi_y^s(z) - \phi_y^{\alpha}(z) \big| \big| f_{1/2}^s(z) - f_{1/2}^{\alpha}(z) \big|\,dz.
	\]
	By \eqref{eq:35} and \cite[(4.16)]{BoPaWa2018} we see that $f_{1/2}^s$ and $f_{1/2}^{\alpha}$ are uniformly bounded by a fixed bounded integrable function $g(z) = 1 \wedge |z|^{-d-\alpha+\delta}$, if only $s$ is large enough. Thus, for every $R>0$ we have, by Lemma \ref{prop:dhk}, \eqref{eq:36} and the dominated convergence theorem,
	\begin{align}\label{eq:43}\begin{aligned}
			\limsup_{s \to \infty} &\int_{\Gamma_R} \big| \phi_y^s(z) - \phi_y^{\alpha}(z) \big| \big| f_{1/2}^s(z) - f_{1/2}^{\alpha}(z) \big|\,dz \\ \leq c &\limsup_{s \to \infty} \int_{\Gamma_R} \big| \phi_y^s(z) - \phi_y^{\alpha}(z) \big|\,dz = 0.\end{aligned}
	\end{align}
	Moreover,
	\begin{equation}\label{eq:44}
		\int_{\Gamma \setminus \Gamma_R} \big| \phi_y^s(z) - \phi_y^{\alpha}(z) \big| \big| f_{1/2}^s(z) - f_{1/2}^{\alpha}(z) \big|\,dz \leq 4c_1\int_{\Gamma \setminus \Gamma_R} g(z)\,dz.
	\end{equation}
	Now, \eqref{eq:43} together with \eqref{eq:44} yield that for every $R>0$,
	\begin{align*}
		\limsup_{s \to \infty} \big| \mu_{1/2}^s \big( \phi_y^s -\phi_y^{\alpha}\big) - \mu_{1/2}^{\alpha} \big( \phi_y^s -\phi_y^{\alpha}\big) \big| \leq 4c_1\int_{\Gamma \setminus \Gamma_R} g(z)\,dz.
	\end{align*}
	By letting $R$ to infinity, we obtain \eqref{eq:45}, so \eqref{eq:42} follows.
	
	We thus have  proved that the left-hand side limit in \eqref{eq:46} exists for all $y \in \Gamma$. Let us denote $\lim_{s \to \infty} f_1^s(y) = \widetilde{n}(y)$ for $y \in \Gamma$. By weak convergence, \eqref{eq:35} and the dominated convergence theorem, for every bounded continuous function $\phi$ we have
	\[
	\int_{\Gamma} n^{\alpha}(y)\phi(y)\,dy = \mu^{\alpha}(\phi) =  \lim_{s \to \infty} \int_{\Gamma} \frac{p_1^{\Gamma,s}(x_s,y)}{\P_{x_s}^s(\tau_{\Gamma}^s>1)}\phi(y)\,dy = \int_{\Gamma} \widetilde{n}(y)\phi(y)\,dy. 
	\]
	Therefore, $n^{\alpha} = \widetilde{n}$ and \eqref{eq:46} follows.
\end{proof}
We are ready to prove the main result of the paper.
\begin{proof}[Proof of Theorem \ref{thm:Yaglom}]	
	Using \eqref{eq:dhks}, \eqref{eq:surv} and changing variables we get, for $x \in \Gamma$ and $s \geq 1$,
	\begin{align*}
		\P_x \big( \psi^{-1}(1/s)X_s \in A | \tau_{\Gamma}>s \big) &= \frac{\P_x \big( \psi^{-1}(1/s)X_s \in A, \tau_{\Gamma}>s  \big)}{\P_x (\tau_{\Gamma}>s)} \\ &= \frac{\int_{A/\psi^{-1}(1/s)}p_s^{\Gamma}(x,y)\,dy}{\int_{\Gamma} p_s^{\Gamma}(x,y)\,dy} \\ &= \frac{\int_{A} p_1^{\Gamma,s}(x_s,y)\,dy}{\int_{\Gamma}p_1^{\Gamma,s}(x_s,y)\,dy} \\ &= \frac{\int_{A} p_1^{\Gamma,s}(x_s,y)\,dy}{\P_{x_s}^s(\tau_{\Gamma}^s>1)} \\ &= \mu_1^s(A).
	\end{align*}
By \eqref{eq:46} and \eqref{eq:35} and an application of \cite[Lemma 4.1]{BoPaWa2018} we conclude that
	\[
	\lim_{s \to \infty} \P_x \big( \psi^{-1}(1/s)X_s \in A | \tau_{\Gamma}>s \big) = \int_A n^{\alpha}(y)\,dy = \mu^{\alpha}(A).
	\]
	The proof is complete.
\end{proof}
\begin{example}\label{e.da}
Consider the
L\'evy density $\nu(x)=\mathcal{A}(d,\alpha)  \ln^\beta(e+|x|)/|x|^{d+\alpha}$ on $\Rd$, with $\beta\in\mathbb{R}$ and, of course, $\alpha\in (0,2)$. Then {\bf A} is satisfied, for instance $\psi(r)\sim r^\alpha \ln^\beta(1/r)$ for small $r>0$, see \cite[Proposition 2]{WCTGBT}. In this case, the rescaling in Theorem \ref{thm:Yaglom} is by $\psi^{-1}(1/s) \sim s^{-1/\alpha} \ln^{-\beta/\alpha}(s^{1/\alpha})$, which is qualitatively different than in \eqref{e.Ysl}.
\end{example}

\appendix

\newcommand{\Q}{\mathbb{Q}}
\section{Weak convergence and continuity in the Skorokhod topology}
This appendix is devoted to weak convergence and continuity in the Skorokhod topology. 

Let us recall the Skorokhod topology and its basic properties. The main references here are the books \cite{JS87} and Gihman and Skorohod \cite{GS80}. 
In view of \cite[Theorem VI.1.14]{JS87}, the space $\calD(\Rd) = \calD_{[0, \infty)}(\R^d)$ of all of c\`{a}dl\`{a}g functions $\omega \colon [0,\infty) \mapsto \R^d$,  may be endowed with a topology for which it is a complete separable metric space. The notion of convergence is described as follows: let $\Lambda$ be a set of strictly increasing continuous functions $\lambda \colon [0,\infty) \mapsto [0,\infty)$ such that $\lambda(0)=0$ and $\lambda(t) \to \infty$ as $t \to \infty$. Then $\omega_n$ converges to $\omega$ in $\calD(\Rd)$ as $n \to \infty$ if and only if there is a sequence $\{\lambda_n\colon  n \geq 1 \} \subset \Lambda$ such that
\begin{equation}
  \label{e:lambdant}
  \lim_{n \to \infty} \sup_{t \in [0,\infty)} |\lambda_n(t)-t|=0,
\end{equation}
and
\begin{equation}
  \label{e:omegant}
  \lim_{n \to \infty} \sup_{t \leq T} \big| \omega_n (\lambda_n(t))-\omega_n \big|=0 \quad \text{for every } T>0. 
\end{equation}

Let $\bfY$ be the canonical projection or a coordinate process in $\calD(\Rd)$, i.e.
\[
Y_t = Y_t(\omega)=\omega(t)
\]
for $t \geq 0$ and $\omega \in \calD(\Rd)$. As before, we define for open set $D\subset \Rd$, 
\[
  \tau_D(\omega) = \inf \{ t > 0 \colon \omega(t) \not\in D \}. 
\]

Let $\P_x$ be a probability measure on $\calD(\Rd)$ such that $\P_x \left( Y_0 = x \right) = 1$. We assume throughout this appendix that the process $\bfY$ is quasi-left-continuous under $\P_x$ and for every $x \in D$, $\P_x(Y_{\tau_D} \in \partial D)=0$.

\begin{proposition}
  \label{prop:continuity}
  Let $\D$ be the set of discontinuities of the functionals $\tau_{D}$,
  $Y_{\tau_{D-}}$ and $Y_{\tau_{D}}$ on $\calD(\R^d)$. Then
  $\P_x(\D) = 0$ for all $x \in D$. 
\end{proposition}
\begin{proof}
  For $m \geq 1$ we define $D_m = \{ x \in D\colon \dist(x,\, D^c) >
  1/m\}$ and $\tau_m = \inf\{ t\colon Y_t \notin D_m \}$. Of course, $\tau_m \leq \tau_D$. If $\tau_{\infty} := \lim_{m \to \infty} \tau_m = \infty$, then 
  $\tau_{\Gamma} = \tau_{\infty}$. On $\{\tau_{\infty} < \infty\}$  we have from
  the quasi-left-continuity of 
  $\bfY$ that $\lim_{m \to \infty} Y_{\tau_m} = Y_{\tau_{\infty}}$ a.s. It follows that
  $Y_{\tau_{\infty}} \in D^c$, hence $\tau_{\infty} \geq
  \tau_D$ and so $\tau_{\infty} =
  \tau_D$.  
  Thus $\lim_{m \to \infty} \tau_m \uparrow \tau_{\Gamma}$. 
  
  By our assumption, $\P_x \left( Y_{\tau_{D} -} \in D,
    Y_{\tau_D} \in \overline{D}^c \right) = 1$. Fix $\omega \in
  \calD(\R^d)$ such that $Y_{\tau_{D}-}(\omega) \in D$, $Y_{\tau_{D}}
  \in \overline{D}^c$ and $\lim_{m \to \infty} \tau_m(\omega) =
  \tau_{D}(\omega)$. We will prove that the functionals $\tau_{D}$, $Y_{\tau_D -}$ and
  $Y_{\tau_{D}}$ are continuous at $\omega$.

  Let $\{ \omega_n \}$ be a sequence in $\calD(\Rd)$ such that $\omega_n \to \omega$ as $n \to \infty$. By the definition of the Skorokhod topology, there exist $\lambda_n \in \Lambda$
  such that $\sup_{s > 0} \left| \lambda_n(s) 
    - s \right| \to 0$ and $\sup_{s \leq T} \left| Y_{\lambda_n(s)}(\omega_n) -
    Y_s(\omega) \right| \to 0$ for every $T > 0$. Let $t = \tau_{D}(\omega)$
  and $t_n = \lambda_n(t)$. Then $t_n \to t$, $Y_{t_n}(\omega_n) \to
  Y_t(\omega)$ and $Y_{t_n-}(\omega_n) \to Y_{t-}(\omega)$ as $n \to \infty$.

  Since $Y_t(\omega) \in \overline{D}^c$, we have
  $Y_{t_n}(\omega_n) \in \overline{D}^c$ so
  $\tau_{D}(\omega_n) \leq t_n$ for $n$ large enough. Therefore,
  \begin{equation}
    \label{e:tauup}
    \limsup_{n\to\infty} \tau_{D}(\omega_n) \leq \tau_{D}(\omega). 
  \end{equation}

  For every $\varepsilon > 0$, there exists $m_0$ such that $t - \varepsilon < \tau_{m_0}(\omega) \leq t$. Note that for $n$ sufficiently large,  $\sup_{u \leq t - \varepsilon} \left| \lambda_n(u) - u \right| < \varepsilon$ and $\sup_{u \leq t - \varepsilon} \left| Y_{\lambda_n(u)}(\omega_n) - Y_u(\omega) \right| < 1/m_0$. Since $Y_u(\omega) \in D_{m_0}$ we have $Y_{\lambda_n(u)}(\omega_n) \in D$ for every $u \leq t - \varepsilon$. Therefore $\tau_D(\omega_n) \geq \lambda_n(t - \varepsilon) \geq t - 2 \varepsilon$. As $\varepsilon$ is arbitrary, $\liminf_{n \to \infty} \tau_D(\omega_n) \geq t$. This and \eqref{e:tauup} prove the continuity of $\tau_D$ at $\omega$. 

Note that we have shown that $\lim_{n \to \infty} \tau_D(\omega_n) = t$ and for $n$ large enough, $\tau_D(\omega_n) \leq t$. Applying \cite[Proposition~VI.2.1 (b.2)]{JS87} we get that $\lim_{n\to\infty} Y_{\tau_{D}-}(\omega_n) = Y_{\tau_{D}-}(\omega)$.

If $\tau_{D}(\omega_{n_k}) < t_{n_k}$ for a subsequence $\{n_k\colon k \in \N \}$, then, by
\cite[Proposition~VI.2.1 (b.1)]{JS87}, we have $\lim_{k\to\infty}
Y_{\tau_{D}}(\omega_{n_k}) = Y_{\tau_{D}-}(\omega)$. This contradicts
the facts that $Y_{\tau_{D}}(\omega_n) \in D^c$ and
$Y_{\tau_{D}-}(\omega) \in D$. So we must have that
$\tau_{D}(\omega_n) = t_n$ for $n$ large enough. Therefore, $\lim_{n\to
  \infty} Y_{\tau_{D}}(\omega_n) = Y_{\tau_{D}}(\omega)$.
\end{proof}

\begin{corollary}
	\label{cor:boundary}
	For every $t > 0$, $\P_x \left( \partial \left\{ \tau_{D} > t
	\right\} \right) = 0$. 
\end{corollary}

\begin{proof}
	The boundary of $\left\{ \tau_{D} > t
	\right\}$ in the Skorokhod topology is contained in $\D \cup \left\{ \tau_{D} = t \right\}$, where
	$\D$ is the set of discontinuities of $\tau_{D}$. By
	Proposition~\ref{prop:continuity} we have that $\P_x(\D) = 0$. Since $Y_{t-} =
	Y_t$ a.s., we have that $\P_x(\tau_{D} = t) \leq
	\P_x(Y_{\tau_{D}} \in \partial D) = 0$. 
\end{proof}

Now assume that, under $\P_x$, $\bfY$ is a pure-jump L\'evy process in $\R^d$ with L\'evy exponent $\psi$. For $s \geq 1$ let $\{ \bfY^s \colon t \geq 0 \}$ be a L\'evy processes with L\'evy exponent $\psi^s$. Set
\[
  \tau_D^s = \inf \{ t > 0 \colon Y^s_t \not\in D \}. 
\]
Suppose that for each $\xi \in \R^d$, $\psi^s(\xi) \to \psi(\xi)$ as $s \to \infty$. Then $Y^s_1$ converges in distribution to $Y_1$. It follows from \cite[Corollary VII.3.6]{JS87} that $\bfY^s$ converges in distribution to $\bfY$ in the Skorokhod space $\calD(\R^d)$ as $s \to \infty$. By Corollary \ref{cor:boundary} we have the following result.
\begin{corollary}
  \label{cor:surv_lim}
  Let $D$ be an open subset of $\R^d$ such that $\P_x \left( Y_{\tau_D} \in \partial D \right) = 0$. Then for every $t > 0$ and $x \in D$,
  \[
    \lim_{s \to \infty} \P_x \left( \tau_D^s > t \right) = \P_x \left( \tau_D > t \right). 
  \]
\end{corollary}

\bibliographystyle{abbrv}
\bibliography{LY}

\begin{thebibliography}{10}

\bibitem{MR3498004}
A.~Asselah, P.~A. Ferrari, P.~Groisman, and M.~Jonckheere.
\newblock Fleming--{V}iot selects the minimal quasi-stationary distribution:
  {T}he {G}alton--{W}atson case.
\newblock {\em Ann. Inst. Henri Poincar\'e Probab. Stat.}, 52(2):647--668,
  2016.

\bibitem{BB04}
R.~Ba{\~n}uelos and K.~Bogdan.
\newblock Symmetric stable processes in cones.
\newblock {\em Potential Anal.}, 21(3):263--288, 2004.

\bibitem{BBLPPT}
N.~Bean, L.~Bright, G.~Latouche, C.~Pearce, P.~Pollett, and P.~Taylor.
\newblock {The quasi-stationary behavior of quasi-birth-and-death processes}.
\newblock {\em The Annals of Applied Probability}, 7(1):134 -- 155, 1997.

\bibitem{BPT98}
N.~Bean, P.~Pollett, and P.~Taylor.
\newblock The quasistationary distributions of level-independent
  quasi-birth-and-death processes.
\newblock {\em Communications in Statistics. Stochastic Models},
  14(1-2):389--406, 1998.

\bibitem{BPT00}
N.~Bean, P.~Pollett, and P.~Taylor.
\newblock Quasistationary distributions for level-dependent
  quasi-birth-and-death processes.
\newblock {\em Communications in Statistics. Stochastic Models},
  16(5):511--541, 2000.

\bibitem{BOP19}
N.~G. Bean, M.~M. O'Reilly, and Z.~Palmowski.
\newblock Yaglom limit for stochastic fluid models.
\newblock {\em Adv. Appl. Probab.}, 53(3):1--35, 2021.

\bibitem{MR1232850}
J.~Bertoin.
\newblock Splitting at the infimum and excursions in half-lines for random
  walks and {L}\'evy processes.
\newblock {\em Stochastic Process. Appl.}, 47(1):17--35, 1993.

\bibitem{MR1331218}
J.~Bertoin and R.~A. Doney.
\newblock On conditioning a random walk to stay nonnegative.
\newblock {\em Ann. Probab.}, 22(4):2152--2167, 1994.

\bibitem{RegularVariation89}
N.~H. Bingham, C.~M. Goldie, and J.~L. Teugels.
\newblock {\em Regular variation}, volume~27 of {\em Encyclopedia of
  Mathematics and its Applications}.
\newblock Cambridge University Press, Cambridge, 1989.

\bibitem{KBTGMR10}
K.~Bogdan, T.~Grzywny, and M.~Ryznar.
\newblock Heat kernel estimates for the fractional {L}aplacian with {D}irichlet
  conditions.
\newblock {\em Ann. Probab.}, 38(5):1901--1923, 2010.

\bibitem{KBTGMR-jfa}
K.~Bogdan, T.~Grzywny, and M.~Ryznar.
\newblock Density and tails of unimodal convolution semigroups.
\newblock {\em J. Funct. Anal.}, 266(6):3543--3571, 2014.

\bibitem{KBTGMR-dhk}
K.~Bogdan, T.~Grzywny, and M.~Ryznar.
\newblock Dirichlet heat kernel for unimodal {L}\'evy processes.
\newblock {\em Stochastic Process. Appl.}, 124(11):3612--3650, 2014.

\bibitem{KBTGMR-ptrf}
K.~Bogdan, T.~Grzywny, and M.~Ryznar.
\newblock Barriers, exit time and survival probability for unimodal {L}\'evy
  processes.
\newblock {\em Probab. Theory Related Fields}, 162(1-2):155--198, 2015.

\bibitem{BKK15}
K.~Bogdan, T.~Kumagai, and M.~Kwa{\'s}nicki.
\newblock Boundary {H}arnack inequality for {M}arkov processes with jumps.
\newblock {\em Trans. Amer. Math. Soc.}, 367(1):477--517, 2015.

\bibitem{BoPaWa2018}
K.~Bogdan, Z.~Palmowski, and L.~Wang.
\newblock Yaglom limit for stable processes in cones.
\newblock {\em Electron. J. Probab.}, 23:1--19, 2018.

\bibitem{MR1419491}
L.~Chaumont.
\newblock Conditionings and path decompositions for {L}\'{e}vy processes.
\newblock {\em Stochastic Process. Appl.}, 64(1):39--54, 1996.

\bibitem{MR2164035}
L.~Chaumont and R.~A. Doney.
\newblock On {L}\'evy processes conditioned to stay positive.
\newblock {\em Electron. J. Probab.}, 10:no. 28, 948--961, 2005.

\bibitem{CKS14}
Z.-Q. Chen, P.~Kim, and R.~Song.
\newblock Dirichlet heat kernel estimates for rotationally symmetric {L}\'{e}vy
  processes.
\newblock {\em Proc. Lond. Math. Soc. (3)}, 109(1):90--120, 2014.

\bibitem{WCTGBT}
W.~Cygan, T.~Grzywny, and B.~Trojan.
\newblock Asymptotic behavior of densities of unimodal convolution semigroups.
\newblock {\em Trans. Amer. Math. Soc.}, 369(8):5623--5644, 2017.

\bibitem{MR3648297}
I.~Czarna and Z.~Palmowski.
\newblock Parisian quasi-stationary distributions for asymmetric {L}\'{e}vy
  processes.
\newblock {\em Statist. Probab. Lett.}, 127:75--84, 2017.

\bibitem{Dynkin65}
E.~B. Dynkin.
\newblock {\em Markov processes. {V}ols. {I}, {II}}, volume 122 of {\em
  Translated with the authorization and assistance of the author by J. Fabius,
  V. Greenberg, A. Maitra, G. Majone. Die Grundlehren der Mathematischen
  Wissenschaften, B\"ande 121}.
\newblock Academic Press Inc., Publishers, New York; Springer-Verlag,
  Berlin-G\"ottingen-Heidelberg, 1965.

\bibitem{MR1334159}
P.~A. Ferrari, H.~Kesten, S.~Martinez, and P.~Picco.
\newblock Existence of quasi-stationary distributions. {A} renewal dynamical
  approach.
\newblock {\em Ann. Probab.}, 23(2):501--521, 1995.

\bibitem{MR2318407}
P.~A. Ferrari and N.~Mari{\'c}.
\newblock Quasi stationary distributions and {F}leming-{V}iot processes in
  countable spaces.
\newblock {\em Electron. J. Probab.}, 12:no. 24, 684--702, 2007.

\bibitem{MR346932}
D.~C. Flaspohler and P.~T. Holmes.
\newblock Additional quasi-stationary distributions for semi-{M}arkov
  processes.
\newblock {\em J. Appl. Probability}, 9:671--676, 1972.

\bibitem{GS80}
I.~I. Gihman and A.~V. Skorohod.
\newblock {\em The theory of stochastic processes. {I}}, volume 210 of {\em
  Grundlehren der Mathematischen Wissenschaften [Fundamental Principles of
  Mathematical Sciences]}.
\newblock Springer-Verlag, Berlin-New York, english edition, 1980.
\newblock Translated from the Russian by Samuel Kotz.

\bibitem{Graves}
L.~M. Graves.
\newblock {\em The Theory of Functions of Real Variables (1st edition)}.
\newblock McGraw-Hill Booc Company, inc., New York, 1946.

\bibitem{TG14}
T.~Grzywny.
\newblock On {H}arnack inequality and {H}\"older regularity for isotropic
  unimodal {L}\'evy processes.
\newblock {\em Potential Anal.}, 41(1):1--29, 2014.

\bibitem{TGMK}
T.~Grzywny and M.~Kwa\'snicki.
\newblock Potential kernels, probabilities of hitting a ball, harmonic
  functions and the boundary {H}arnack inequality for unimodal {L}\'evy
  processes.
\newblock {\em Stochastic Process. Appl.}, 128(1):1--38, 2018.

\bibitem{TGLLMS21}
T.~Grzywny, L.~Le{\.{z}}aj, and M.~Mi\'{s}ta.
\newblock Hitting probabilities for l\'{e}vy processes on the real line.
\newblock {\em ALEA, Lat. Am. J. Probab. Math. Stat.}, 18:727--760, 2021.

\bibitem{MR3007664}
T.~Grzywny and M.~Ryznar.
\newblock Potential theory of one-dimensional geometric stable processes.
\newblock {\em Colloq. Math.}, 129(1):7--40, 2012.

\bibitem{TGMRBT}
T.~{Grzywny}, M.~{Ryznar}, and B.~{Trojan}.
\newblock {Asymptotic behaviour and estimates of slowly varying convolution
  semigroups}.
\newblock {\em Int. Res. Math. Notices}, 2019(23):7193--7258, 2019.
\newblock doi:10.1093/imrn/rnx324.

\bibitem{TGKS19}
T.~Grzywny and K.~Szczypkowski.
\newblock L\'{e}vy processes: concentration function and heat kernel bounds.
\newblock {\em Bernoulli}, 26(4):3191--3223, 2020.

\bibitem{TGKS17}
T.~{Grzywny} and K.~{Szczypkowski}.
\newblock {Estimates of heat kernels of non-symmetric L\'{e}vy processes}.
\newblock {\em Forum Math.}, 33(5):1207--1236, 2021.

\bibitem{HassRivero12}
B.~Haas and V.~Rivero.
\newblock Quasi-stationary distributions and {Y}aglom limits of self-similar
  {M}arkov processes.
\newblock {\em Stochastic Process. Appl.}, 122(12):4054--4095, 2012.

\bibitem{HHKW}
S.~C. Harris, E.~Horton, A.~E. Kyprianou, and M.~Wang.
\newblock Yaglom limit for critical neuron transport.
\newblock Preprint, 2021. ArXiv:2103.02237v2.

\bibitem{MR368168}
D.~L. Iglehart.
\newblock Random walks with negative drift conditioned to stay positive.
\newblock {\em J. Appl. Probability}, 11:742--751, 1974.

\bibitem{JackaRoberts95}
S.~D. Jacka and G.~O. Roberts.
\newblock Weak convergence of conditioned processes on a countable state space.
\newblock {\em J. Appl. Probab.}, 32(4):902--916, 1995.

\bibitem{JS87}
J.~Jacod and A.~N. Shiryaev.
\newblock {\em Limit theorems for stochastic processes}, volume 288 of {\em
  Grundlehren der Mathematischen Wissenschaften [Fundamental Principles of
  Mathematical Sciences]}.
\newblock Springer-Verlag, Berlin, 1987.

\bibitem{MR3010850}
V.~Knopova and R.~L. Schilling.
\newblock A note on the existence of transition probability densities of
  {L}\'{e}vy processes.
\newblock {\em Forum Math.}, 25(1):125--149, 2013.

\bibitem{MR3413864}
T.~Kulczycki and M.~Ryznar.
\newblock Gradient estimates of harmonic functions and transition densities for
  {L}\'{e}vy processes.
\newblock {\em Trans. Amer. Math. Soc.}, 368(1):281--318, 2016.

\bibitem{MK17}
M.~Kwa\'{s}nicki.
\newblock Ten equivalent definitions of the fractional {L}aplace operator.
\newblock {\em Fract. Calc. Appl. Anal.}, 20(1):7--51, 2017.

\bibitem{KJ17}
M.~Kwa\'snicki and T.~Juszczyszyn.
\newblock Martin kernels for {M}arkov processes with jumps.
\newblock {\em Potential Anal.}, 47(3):313--335, 2017.

\bibitem{KyprianouPalmowski06}
A.~E. Kyprianou and Z.~Palmowski.
\newblock Quasi-stationary distributions for {L}\'{e}vy processes.
\newblock {\em Bernoulli}, 12(4):571--581, 2006.

\bibitem{AEK71}
E.~K. Kyprianou.
\newblock On the quasi-stationary distribution of the virtual waiting time in
  queues with {P}oisson arrivals.
\newblock {\em J. Appl. Probability}, 8:494--507, 1971.

\bibitem{MR2299923}
A.~Lambert.
\newblock Quasi-stationary distributions and the continuous-state branching
  process conditioned to be never extinct.
\newblock {\em Electron. J. Probab.}, 12:no. 14, 420--446, 2007.

\bibitem{MR2959448}
M.~Mandjes, Z.~Palmowski, and T.~Rolski.
\newblock Quasi-stationary workload in a {L}\'evy-driven storage system.
\newblock {\em Stoch. Models}, 28(3):413--432, 2012.

\bibitem{SMJS94}
S.~Mart\'{\i}nez and J.~San~Mart\'{\i}n.
\newblock Quasi-stationary distributions for a {B}rownian motion with drift and
  associated limit laws.
\newblock {\em J. Appl. Probab.}, 31(4):911--920, 1994.

\bibitem{MR4171931}
Z.~Palmowski and M.~Vlasiou.
\newblock Speed of convergence to the quasi-stationary distribution for
  {L}\'{e}vy input fluid queues.
\newblock {\em Queueing Syst.}, 96(1-2):153--167, 2020.

\bibitem{Pollet08}
P.~Pollett.
\newblock Quasi-stationary distributions: A bibliography.
\newblock Available at
  \url{https://people.smp.uq.edu.au/PhilipPollett/papers/qsds/qsds.pdf}.

\bibitem{Pruitt81}
W.~E. Pruitt.
\newblock The growth of random walks and {L}\'evy processes.
\newblock {\em Ann. Probab.}, 9(6):948--956, 1981.

\bibitem{MR3841403}
Y.-X. Ren, R.~Song, and Z.~Sun.
\newblock A 2-spine decomposition of the critical {G}alton-{W}atson tree and a
  probabilistic proof of {Y}aglom's theorem.
\newblock {\em Electron. Commun. Probab.}, 23:Paper No. 42, 12, 2018.

\bibitem{MR4102269}
Y.-X. Ren, R.~Song, and Z.~Sun.
\newblock Limit theorems for a class of critical superprocesses with stable
  branching.
\newblock {\em Stochastic Process. Appl.}, 130(7):4358--4391, 2020.

\bibitem{Sato99}
K.-i. Sato.
\newblock {\em L\'evy processes and infinitely divisible distributions},
  volume~68 of {\em Cambridge Studies in Advanced Mathematics}.
\newblock Cambridge University Press, Cambridge, 1999.
\newblock Translated from the 1990 Japanese original, Revised by the author.

\bibitem{MR3644418}
R.~L. Schilling.
\newblock {\em Measures, integrals and martingales}.
\newblock Cambridge University Press, Cambridge, second edition, 2017.

\bibitem{MR207047}
E.~Seneta and D.~Vere-Jones.
\newblock On quasi-stationary distributions in discrete-time {M}arkov chains
  with a denumerable infinity of states.
\newblock {\em J. Appl. Probability}, 3:403--434, 1966.

\bibitem{Sztonyk00}
P.~Sztonyk.
\newblock On harmonic measure for {L}\'evy processes.
\newblock {\em Probab. Math. Statist.}, 20(2, Acta Univ. Wratislav. No.
  2256):383--390, 2000.

\bibitem{Tweedie74}
R.~L. Tweedie.
\newblock Quasi-stationary distributions for {M}arkov chains on a general state
  space.
\newblock {\em J. Appl. Probability}, 11:726--741, 1974.

\bibitem{MR1133722}
E.~A. van Doorn.
\newblock Quasi-stationary distributions and convergence to quasi-stationarity
  of birth-death processes.
\newblock {\em Adv. in Appl. Probab.}, 23(4):683--700, 1991.

\bibitem{MR2240700}
H.~\v{S}iki\'{c}, R.~Song, and Z.~Vondra\v{c}ek.
\newblock Potential theory of geometric stable processes.
\newblock {\em Probab. Theory Related Fields}, 135(4):547--575, 2006.

\bibitem{Yaglom47}
A.~M. Yaglom.
\newblock Certain limit theorems of the theory of branching random processes.
\newblock {\em Doklady Akad. Nauk SSSR (N.S.)}, 56:795--798, 1947.

\bibitem{MR3247530}
J.~Zhang, S.~Li, and R.~Song.
\newblock Quasi-stationarity and quasi-ergodicity of general {M}arkov
  processes.
\newblock {\em Sci. China Math.}, 57(10):2013--2024, 2014.

\end{thebibliography}

\end{document}